\documentclass[12pt,oneside,reqno]{amsart}
\usepackage[latin1]{inputenc}
\usepackage[english]{babel}
\usepackage{amsmath}
\usepackage{amssymb}
\usepackage{enumitem}
\usepackage[paper=a4paper]{geometry}
\usepackage{mathrsfs} 
\usepackage{mathtools}
\usepackage{hyperref}
\usepackage{color}

\usepackage[square,numbers]{natbib}
\author{Henri Elad Altman}
\address{Sorbonne Universit\'e, CNRS, Laboratoire de Probabilit\'es Statistique et Mod\'elisation (LPSM), F-75005 Paris, France}
\email{eladaltman@lpsm.paris}

\title{Bessel SPDEs with general Dirichlet boundary conditions}
\date{\today}

\newfont{\indic}{bbmss12}

\newcommand{\R}{\mathbb R}

\renewcommand{\d}{\, \mathrm{d}}

\theoremstyle{plain}
\newtheorem{thm}{Theorem}[section]
\newtheorem{lm}[thm]{Lemma}
\newtheorem{prop}[thm]{Proposition} 
\newtheorem{cor}[thm]{Corollary} 
  
\theoremstyle{definition}
\newtheorem{df}[thm]{Definition}

\newtheorem{rk}[thm]{Remark}
\numberwithin{equation}{section}
\begin{document}
\maketitle


\begin{abstract}
We generalise the integration by parts formulae obtained in \cite{EladAltman2019} to Bessel bridges on $[0,1]$ with arbitrary boundary values, as well as Bessel processes with arbitrary initial conditions. This allows us to write, formally, the corresponding dynamics using renormalised local times, thus extending the Bessel SPDEs of \cite{EladAltman2019} to general Dirichlet boundary conditions. We also prove a dynamical result for the case of dimension $2$, by providing a weak construction of the gradient dynamics corresponding to a $2$-dimensional Bessel bridge. 
\end{abstract}

\section{Introduction}

\subsection{Bessel SPDEs}

The purpose of this paper is to further extend the results obtained recently in \cite{EladAltman2019}, and which introduced Bessel SPDEs of dimension smaller than $3$.
 
Bessel processes are a one-parameter family of nonnegative real-valued diffusions which play a central role in various fields, ranging from statistical mechanics to finance. 
From the perspective of stochastic analysis, they appear naturally in the study of Brownian motion, see e.g. sections VI.3 and XI.2 in \cite{revuz2013continuous}, but they also provide a highly non-trivial example of stochastic process for which the theory of stochastic calculus due to Kiosy It\^{o} allows to derive numerous remarkable results.  Recall that, for any $\delta \geq 0$, a $\delta$-dimensional Bessel process $(\rho_t)_{t\geq 0}$ is defined as $\rho_t = \sqrt{X}_t$, where $(X_t)_{t \geq 0}$ is a $\delta$-dimensional squared Bessel process, which is in turn the unique, nonnegative, solution to the SDE
\[X_t=X_0+\int_0^t2\sqrt{X_s} \d B_s +\delta \,t, \quad t\geq 0, \]
(see Chapter XI in \cite{revuz2013continuous}). Then $(\rho_t)_{t \geq 0}$ is itself the solution to some SDE with a singular drift. Namely, for $\delta>1$, $\rho$ is the solution to
\begin{equation}\label{sde1}
\rho_t=\rho_0+\frac{\delta-1}2\int_0^t \frac1{\rho_s}\d s+ B_t, \quad t\geq 0. \qquad (\delta>1)
\end{equation}
By contrast, for $\delta=1$, $\rho := \sqrt{X}$ is the solution to
\[
\rho_t=\rho_0+ L_t + B_t, \qquad \quad t\geq 0, \qquad (\delta=1)
\]
where $(L_t)_{t\geq 0}$ is continuous and monotone non-decreasing, with $L_0=0$ and
\begin{equation}\label{sde2}
\rho\geq 0, \qquad \int_0^\infty \rho_s \d L_s=0.
\end{equation}
In other words $\rho$ is a reflecting Brownian motion. However, for $\delta\in (0,1)$, the equation solved by $\rho:=\sqrt{X}$ is substantially more difficult. Indeed, in that case, $\int_0^t \frac{1}{\rho_s} \d s = \infty $ almost-surely, and the SDE for $\rho$ can be formally written using renormalisation
\[ \rho_t=\rho_0+\frac{\delta-1}2 \left(\int_0^t \frac1{\rho_s}\d s - \infty \right) + B_t. \]
This is reminiscent of renormalisations that arise in the context of singular stochastic PDEs and which have recently gained much attention with the development of the theories of regularity structures and paracontrolled distributions allowing to analyse such equations. However, the kind of renormalisation entering here into play is quite different from the schemes generally used in these theories, since it applies to local times of the solution rather than the solution itself. Namely, one can show - see e.g. \cite[Proposition 3.12]{zambotti2017random} - that 
$X$ admits {\it diffusion local times}, that is a continuous process $(\ell^a_t)_{t\geq 0,a\geq 0}$ such that, a.s.
\[
\int_0^t \varphi(\rho_s)\d s =\int_0^\infty \varphi(a) \, \ell^a_t \, a^{\delta-1} \d a,
\]
for all Borel $\varphi:\mathbb{R}_+\to\mathbb{R}_+$. Then the rigorous SDE for $\rho$ is given by
\begin{equation}\label{sde3}
\rho_t=\rho_0+\frac{\delta-1}2\int_0^\infty\frac{\ell^a_t-\ell^0_t}a \, a^{\delta-1}\d a+B_t, \quad t\geq 0, \qquad (0<\delta<1).
\end{equation}
Thus, even in this relatively simple SDE context, a highly non-trivial renormalisation phenomenon appears, which cannot yet be understood within the framework of the recent pathwise theories mentioned above. More generally, the dynamics of Bessel processes, and notably their reflection mechanism, possess a remarkable richness that makes them an object of particular interest. For instance, these properties play an important role in applications to the study of Schramm-Loewner Evolution, see \cite{katori2016bessel}. As mentioned above, to study these subtle properties, one in general resorts to the theory of stochastic calculus. Unfortunately, such tools typically break down in the context of stochastic PDEs (SPDEs) driven by space-time white noise. Recently, the theories of regularity structures and paracontrolled distributions have created novel tools to study such SPDEs, allowing - among other things - to obtain several results in the spirit of stochastic calculus, see e.g. \cite{Bellingeri}. However, at this point, such results lack some additional identification theorems to be as powerful as those of classical It\^{o}'s theory. This fact motivates the following question: is there, in the world of SPDEs, an analog of Bessel processes? What are their properties and how could one prove them? 

In the series of articles \cite{Z01,zambotti2002integration,zambotti2003integration,zambotti2004occupation} 
Zambotti constructed a family of SPDEs with properties very similar to Bessel processes, with an even richer behavior. More precisely, given $\delta>3$ and a boundary condition $a \geq 0$, the associated {\it Bessel SPDE} is given by
\begin{equation}\label{spde>3}
\left\{ \begin{array}{ll}
{\displaystyle
\frac{\partial u}{\partial t}=\frac 12
\frac{\partial^2 u}{\partial x^2}
 + \frac {\kappa(\delta)}{2 \, u^3} 
 + \xi 
}
\\ \\
u(t,0)=u(t,1)=a, \quad t\geq 0
\end{array} \right.
\qquad \qquad (\delta>3)
\end{equation}
where $u\geq 0$ is continuous and $\xi$ is a space-time white noise on $\mathbb{R}_+\times[0,1]$, and where we have set
\begin{equation}\label{kappadelta}
\kappa(\delta) := \frac{(\delta-3)(\delta-1)}{4}.
\end{equation}
As $\delta\downarrow 3$, the solution to \eqref{spde>3} turns out to converge to the solution of the Nualart-Pardoux equation \cite{nualart1992white}
\begin{equation}\label{spde=3}
\left\{ \begin{array}{ll}
{\displaystyle
\frac{\partial u}{\partial t}=
\frac 12\frac{\partial^2 u}{\partial x^2}
 + \eta+ \xi }
\\ \\
u\geq 0, \ d\eta\geq 0, \
\int_{\mathbb{R}_+\times[0,1]} u\, d\eta=0,
\\ \\
u(t,0)=u(t,1)=a
\end{array} \right.
\qquad \qquad (\delta=3)
\end{equation}
where $\eta$ is a reflection measure on $]0,\infty[\,\times\,]0,1[$. Moreover, the unique invariant probability measure of \eqref{spde>3} for $\delta>3$ (resp. of \eqref{spde=3}) corresponds to the law of a $\delta$-dimensional (resp. $3$-dimensional) Bessel bridge from $a$ to $a$ on the interval $[0,1]$. 
In particular, the SPDE \eqref{spde=3} with $a=0$ admits the law of a normalised Brownian excursion as invariant measure. The above SPDEs arise naturally as scaling limits of discrete random interface models. Thus, equation \eqref{spde=3}  describes the fluctuations of an effective $(1+1)$ interface model \cite{funakiolla,funakistflour} and of weakly asymmetric interfaces \cite{etheridge2015scaling} near a wall, while \eqref{spde>3} describes the fluctuations of interface models with repulsion from a wall \cite{zambotti2004fluctuations}. 
While the SPDEs \eqref{spde>3} for $\delta>3$ are the analog of the SDEs \eqref{sde1}, the SPDE \eqref{spde=3} is the analog of the SDE \eqref{sde2}: see the introduction of \cite{EladAltman2019} for a development of this idea.

One may ask whether the previous SPDEs can be extended in a natural way to the region $\delta<3$. Namely, can one construct SPDEs which possess the laws of Bessel bridges of dimension $\delta<3$ as invariant measure? This question is further motivated by a major open problem: the description of the scaling limit of dynamical critical pinning models, which we conjecture to correspond to the SPDE associated with $\delta=1$. We refer to \cite{dgz} and \cite{deuschel2018scaling} for the study of pinning models, and to Section 15.2 of \cite{funakistflour}, to \cite{fattler2016construction}, \cite{grothaus18feller}, as well as Section 1.5 in \cite{deuschel2018scaling} for constructions of the corresponding dynamics in the discrete setting. The value $\delta=2$ is also of interest, since it corresponds to a transition in the behavior of Bessel processes at $0$, see e.g. Prop. 3.6 in  \cite{zambotti2017random}. We expect that a similar transition should happen at the level of the SPDEs for $\delta=2$, see Section \ref{sect_dynamics} below.
However, for several years, extending the above SPDEs to $\delta<3$, even at a heuristic level, had remained a very open problem. 

\subsection{Extension to $\delta < 3$ : the case of homogeneous Dirichlet boundary conditions} 

The recent article \cite{EladAltman2019} has identified the candidiates for the SPDEs that should correspond to $\delta<3$, in the case of homogeneous Dirichlet boundary conditions. The method used there relies on integration by parts formulae (IbPFs) for the law of Bessel bridges of dimension $\delta<3$ from $0$ to $0$ over the interval $[0,1]$. These formulae imply that, in the case of homogeneous Dirichlet boundary conditions, the Bessel SPDE for $1<\delta<3$ should have the following form
\begin{equation}\label{1<spde<3}
\frac{\partial u}{\partial t}=\frac 12
\frac{\partial^2 u}{\partial x^2}
 + \frac {\kappa(\delta)}{2}\frac\partial{\partial t}\int_0^\infty \frac1{b^3}\left(\ell^b_{t,x}-\ell^0_{t,x}\right)  b^{\delta-1}\d b  + \xi,
\quad \quad (1<\delta<3)
\end{equation}
with the boundary condition $u(t,0)=u(t,1)=0$ and where, for all $x \in (0,1)$, the local time process $(\ell^b_{t,x})_{b\geq 0, t \geq 0}$ is defined by
\begin{equation}
\label{occupation_time_formula}
\int_0^t \varphi(u(s,x))\d s =\int_0^\infty \varphi(b) \, \ell^b_{t,x} \, b^{\delta-1} \d b,
\end{equation}
for all Borel $\varphi:\mathbb{R}_+\to\mathbb{R}_+$. 
Then \eqref{1<spde<3} appears as an analog, in the context of SPDEs, of the SDE \eqref{sde3}. For $\delta=1$, the SPDE should be
\begin{equation}\label{spde=1}
\frac{\partial u}{\partial t}=\frac 12
\frac{\partial^2 u}{\partial x^2}
 - \frac{1}{8} \frac\partial{\partial t}\frac{\partial^{2}}{\partial b^{2}} \, \ell^{b}_{t,x}\, \biggr\rvert_{b=0} + \xi
\quad, \quad (\delta=1)
\end{equation}
while for $0<\delta<1$, it should take the form
\begin{equation}\label{0<spde<1}
\begin{split}
& \frac{\partial u}{\partial t}=\ \frac 12
\frac{\partial^2 u}{\partial x^2} + \xi
\qquad\qquad\qquad\qquad\qquad\qquad\qquad (0<\delta<1)
\\ & + \frac {\kappa(\delta)}{2}\frac\partial{\partial t}\int_0^\infty \frac1{b^3}\left(\ell^b_{t,x}-\ell^0_{t,x}-\frac{b^2}2\frac{\partial^{2}}{\partial b^{2}} \, \ell^{b}_{t,x}\, \biggr\rvert_{b=0}\right)  b^{\delta-1}\d b.
\end{split}
\end{equation}
imposing again $u(t,0)=u(t,1)=0$ in both equations. In \cite{EladAltman2019}, the name Bessel SPDEs  was proposed to refer to all these equations. We stress that, in contrast to the corresponding IbPFs, the Bessel SPDEs of parameter $\delta<3$ are, for now, mostly conjectural. However, at a heuristic level, one recognizes in all these equations a common structure. 
Indeed, one can reformulate all the Bessel SPDEs in a unified way. To do so, we first recall the definition of a family of distributions used in \cite{EladAltman2019}: for all $\alpha\in\R$, we define the Schwartz distribution $\mu_{\alpha}$ on $[0,\infty)$ as follows
\begin{itemize}
\item if $ \alpha = -k$ with $k \in \mathbb{N}$, then 
\begin{equation}
\label{integer_alpha}
\langle \mu_{\alpha}, \varphi \rangle := (-1)^{k} \varphi^{(k)}(0), \qquad \forall \,
\varphi \in C^\infty_0([0,\infty))
\end{equation}
\item otherwise, 
\begin{equation}
\label{non_integer_alpha}
\langle \mu_{\alpha} , \varphi \rangle := \int_{0}^{+ \infty} \left(
\varphi(b) - \sum_{0\leq j\leq 
-\alpha} \frac{b^{j}}{j!} \, \varphi^{(j)}(0) \right) \frac{b^{\alpha -1}}{\Gamma(\alpha)} \d b, \quad \forall \,
\varphi \in C^\infty_0([0,\infty)).
\end{equation}
\end{itemize}
Then the function $\alpha\mapsto\langle \mu_{\alpha} , \varphi \rangle$ is analytic for all $\varphi \in C^\infty_0([0,\infty))$ (see Section 3.5 in \cite{MR3469458}). Moreover, as noted in the introduction of \cite{EladAltman2019}, for all $\delta>0$ the non-linearity in \eqref{spde>3}- \eqref{spde=3}-\eqref{1<spde<3}-\eqref{spde=1}-\eqref{0<spde<1} can be expressed as
\[\frac{\Gamma (\delta)}{8(\delta-2)} \langle \mu_{\delta-3},\ell^{\boldsymbol{\cdot}}_{t,x}\rangle, \]
noting that the singularity at $\delta=2$ is only apparent due to the conjectured vanishing property
\[ \langle \mu_{-1},\ell^{\boldsymbol{\cdot}}_{t,x}\rangle =0.\]
Thus, formally, the drift term of equations \eqref{1<spde<3}, \eqref{spde=1} and \eqref{0<spde<1} corresponds to the unique analytic continuation, to the region $\delta <3$, of the drift term of the SPDEs \eqref{spde>3}.  While the well-posedness of these equations remains conjectural (see in particular the discussion in Section 6 of \cite{EladAltman2019}), for the case $\delta=1$, \cite{EladAltman2019} proved the existence of a weak solution using Dirichlet Forms methods. Namely, with Dirichlet Forms one can construct a Markov process $(u_t)_{t \geq 0}$ with the law of the modulus of a Brownian bridge as reversible measure (a construction already done in Section 5 of \cite{vosshallthesis}), and it was proved in \cite{EladAltman2019} that this process, at equilibrium, satisfies a weak version of equation \eqref{spde=1} above. More precisely, it was shown that
\begin{equation}\label{formal1}
\frac{\partial u}{\partial t}=\frac 12
\frac{\partial^2 u}{\partial x^2}
 - \frac{1}{4} \, \underset{\epsilon \to 0}{\lim} \, \rho''_{\epsilon}(u) + \xi, 
 \qquad (\delta=1)
\end{equation}
where  $\rho_{\epsilon} = \frac{1}{\epsilon} \rho(\frac{x}{\epsilon})$ is a smooth approximation of the Dirac measure at $0$, see Theorem 5.9 in that article for the precise statement. 

\subsection{Our results}

In this article, we extend the results of \cite{EladAltman2019} in two directions. As a first enhancement, we extend the integration by parts fomulae obtained in \cite{EladAltman2019} to Bessel processes on $[0,1]$ with arbitrary initial condition, as well as Bessel bridges with arbitrary boundary values.
These formulae extend naturally the results of \cite{EladAltman2019} which were restricted to $a=a'=0$, i.e. to homogeneous Dirichlet boundary conditions at the level of the SPDEs. Therefore, we conjecture that the natural extension of the SPDEs \eqref{spde>3} and \eqref{spde=3} to the region $\delta<3$ is given by the SPDEs \eqref{1<spde<3}, \eqref{spde=1} and \eqref{0<spde<1} containing renormalised local times, but with general Dirichlet boundary conditions $a \geq 0$ instead of homogeneous ones. In other words, the structure of the Bessel SPDEs unveiled in \cite{EladAltman2019} is preserved in the case of general Dirichlet boundary conditions: only the boundary conditions have to be adjusted accordingly. This also bears out the idea that the appearance of renormalised local times in the drift term of the SPDEs observed in \cite{EladAltman2019} is an inherent feature of these equations rather than an artefact 
specific to one particular boundary condition. 
A second enhancement of the results of \cite{EladAltman2019} provided here is a dynamic one. Namely, exploiting the IbPF for $\delta=2$, we provide a construction of the dynamics corresponding to the law of a $2$-dimensional Bessel bridge from $0$ to $0$ on the interval $[0,1]$ using Dirichlet form techniques. This generalises, to the case $\delta=2$, 
the result of Section 5 of \cite{EladAltman2019} for the case $\delta=1$.   

\subsubsection{Integration by parts formulae for the laws of Bessel bridges}

As said above, our main tools to investigate the SPDEs presented above are integration by parts formulae (IbPFs). 
%
%
Already in the articles \cite{zambotti2002integration} and \cite{zambotti2003integration} dedicated to the study of the SPDEs \eqref{spde=3} and \eqref{spde>3}, Zambotti had derived IbPFs for the corresponding invariant probability measures, the laws of Bessel bridges of dimension $\delta \geq 3$ on $[0,1]$. In that case, the SPDEs could be solved using the technique introduced by Nualart and Pardoux in \cite{nualart1992white} as well as monotonicity arguments based on the dissipativity of the drift. The IbPFs were needed to derive fine properties of the solution as was done for instance in \cite{zambotti2004occupation}, or for the study of the reflection measure $\eta$ appearing in \eqref{spde=3}.

On the other hand, in the regime $\delta <3$, the classical tools based on monotonicity of the drift break down, and, prior to \cite{EladAltman2019}, it was not even clear what a good candidate for the drift in the SPDE should be. For the moment, the only approach we have at our disposal  to tackle these SPDEs consists in deriving IbPFs for the corresponding invariant measures, that is the laws of Bessel bridges. However, extending the IbPFs obtained in \cite{zambotti2002integration} and \cite{zambotti2003integration} to that regime is highly non-trivial. 
Indeed, while the laws of Bessel bridges of dimension $\delta \geq 3$ can be represented as Gibbs measures with an explicit, convex potential with respect to the law of a Brownian bridge (cf. Prop 3.23 in \cite{zambotti2017random}), this is no longer the case when $\delta<3$. This is why IbPFs in the regime $\delta <3$ had remained out of reach for several years. An exception was the case $\delta=1$, corresponding to the law of the reflected Brownian bridge on $[0,1]$, and for which IbPFs had been obtained using Gaussian calculus: see \cite{zambotti2005integration} and \cite{grothaus2016integration}.  

The problem of extending the IbPFs to $\delta<3$ was solved in \cite{EladAltman2019} which derived IbPFs for the laws of $\delta$-dimensional Bessel bridges from $0$ to $0$ over the interval $[0,1]$ for any $\delta<3$, thus opening the way to the SPDEs \eqref{1<spde<3} \eqref{spde=1} and \eqref{0<spde<1} mentioned above. The computations leading to these formulae rely on semi-explicit formulae for Laplace transforms of squared Bessel bridges from $0$ to $0$ over $[0,1]$, which are consequences of the additivity property of squared Bessel processes first observed by Shiga and Watanabe in \cite{shiga1973bessel}. More precisely, let us henceforth denote by $C([0,1]) := C([0,1], \mathbb{R})$ the space of continuous real-valued function on $[0,1]$. In \cite{EladAltman2019}, an important role was played by the vector space $S$ generated by all functionals on $C([0,1])$ of the form
\begin{align}
\label{exp_functional}
\begin{cases}
C([0,1]) \to \mathbb{R} \\
X \mapsto \exp \left( - \langle m, X^{2} \rangle \right),
\end{cases}
\end{align}
where $m$ is a finite Borel measure on $[0,1]$ and $\langle m, X^{2} \rangle := \int_0^1 X_t^{2} \, m({\rm{d}} t) $. 
In fact, as a consequence of the additivity property of squared Bessel processes, such functionals act on the laws of Bessel processes as a Girsanov transformation corresponding to a deterministic time-change (see Lemma 3.3 in \cite{EladAltman2019}), thus allowing to perform semi-explicit computations, along the lines of Chapter XI of \cite{revuz2013continuous}. In particular, one has the following remarkable formula for the Laplace transform of the square of a $\delta$-dimensional Bessel bridge conditioned to hit $a$ at a time $r \in (0,1)$ 
\begin{equation}
\label{nice_expression}
\begin{split}
& E^{\delta} [\exp (- \langle m, X^{2} \rangle) \, | \, X_{r} = a] \\
& = \exp \left(-\frac{a^{2}}{2} \left( \frac{\psi_1}{\psi_r \hat{\psi}_r} - \frac{1}{r(1-r)} \right)\right)  \left(\frac{r(1-r)}{\psi_r \hat{\psi}_r}\right)^{\delta/2},
\end{split}
\end{equation}
see (3.18) in \cite{EladAltman2019}. Note the multiplicative structure of the above quantity, with the separation of the dependance on $a$ on the one side, and the dependance on $\delta$ on the other. We stress that such formulae are classical and were exploited in several contexts, see e.g. Theorem (3.2) in Chapter XI of \cite{revuz2013continuous}.  The main novelty in \cite{EladAltman2019} was to exploit these convenient identities to derive IbPFs for the functionals in $S$ with respect to the laws $P^{\delta}$ of $\delta$-dimensional Bessel bridges from $0$ to $0$ over $[0,1]$, for any $\delta > 0$, see Theorem 4.1 in \cite{EladAltman2019}. We emphasise that, while \eqref{nice_expression} yields the Laplace transform of a conditioned \textit{squared} Bessel bridge, to our knowledge, there is no general formula for the Laplace transform of Bessel bridges: if one replaces in the left-hand side of \eqref{nice_expression} the $X^2$ by a $X$, one cannot hope to still have an explicit expression in the right-hand side, hence the importance of considering the space $S$ rather than more classical spaces of functionals. Thus, to some extent, functionals of the type \eqref{exp_functional} play, in this context, the same role as functionals of the form $ \exp \left( \langle k, X \rangle \right)$, $k \in C([0,1])$, in the papers \cite{zambotti2005integration} and \cite{grothaus2016integration}, where $\langle \cdot , \cdot \rangle$ denotes the $L^2$ inner product on $[0,1]$.



In this article, we extend the IbPFs to Bessel bridges with arbitrary boundary values. For all $\delta >0$ and $a,a' \geq 0$, let $P^{\delta}_{a,a'}$ be the law, on $C([0,1])$, of a $\delta$-dimensional Bessel bridge between $a$ and $a'$. For all $b \geq 0$ and $r\in(0,1)$, let moreover $\Sigma^{\delta,r}_{a,a'}({\rm d}X \,|\, b)$ denote the finite measure on $C([0,1])$ given by
\begin{equation*}
\Sigma^{\delta,r}_{a,a'}({\rm d}X \,|\, b) := \frac{p^{\delta,r}_{a,a'}(b)}{b^{\delta-1}} \,
 P^{\delta}_{a,a'} [ \d X | \, X_{r} = b].
\end{equation*}
In the above, for all $r \in (0,1)$, $p^{\delta,r}_{a,a'}$ denotes the density of the law of $X_{r}$ under $P^{\delta}_{a,a'}$. 
%
The measure $\Sigma^{\delta,r}_{a,a'}(\,\cdot \,|\, b)$ is meant to be the Revuz measure of the diffusion local time of $(u(t,r))_{t\geq 0}$ at level $b$, with $u(t, \cdot)_{t \geq 0}$ a conjectural infinite-dimensional diffusion with invariant measure $P^{\delta}_{a,a'}$. Note that, for all $r \in (0,1)$ and $b \geq 0$, $\Sigma^{\delta,r}_{0,0}({\rm d}X \,|\, b)$ coincides with the measure $\Sigma^{\delta}_{r}({\rm d}X \,|\, b)$ of Def 3.4 in \cite{EladAltman2019}. As in \cite{EladAltman2019}, we use a convenient notation: for any sufficiently differentiable function $f: \mathbb{R}_{+} \to \mathbb{R}$, for all $n \in \mathbb{Z}$, and all $b \geq 0$, we set
\[ \mathcal{T}^{\,n}_{b} f := f(b) - \sum_{0\leq j\leq n} \frac{b^{j}}{j!} \, f^{(j)}(0). \]
In words, for all $b \geq 0$, if $n\geq 0$ then $\mathcal{T}^{\,n}_{b} f$ is the Taylor remainder centered at $0$, of order $n+1$, of the function $f$, evaluated at $b$; if $n<0$ then $\mathcal{T}^{\,n}_{b} f$ is simply the value of $f$ at $b$. Finally, defining for all $\delta > 0$
\[ \kappa(\delta) := \frac{(\delta-1)(\delta-3)}{4}, \]
and setting 
\[ k:= \left\lfloor \frac{3-\delta}{2} \right\rfloor\leq 1, \]
the IbPFs we obtain in this article can be written as follows. Let $E^{\delta}_{a,a'}$ denote the expectation operator corresponding to the probability measure  $P^{\delta}_{a,a'}$ on $C([0,1])$. Then, for all $\delta \in (0,3)\setminus\{1\}$, $\Phi \in S$ and $h \in C^{2}_{c}(0,1)$, it holds 
  
%

\begin{equation}\label{onetothree}
\begin{split}
& E^{\delta}_{a,a'} (\partial_{h} \Phi (X) ) + E^{\delta}_{a,a'} (\langle h '' , X \rangle \, \Phi(X) )= \\ 
& -\kappa(\delta)\int_{0}^{1}  
 h_{r}  \int_0^\infty b^{\delta-4} \Big[ \mathcal{T}^{\,2k}_{b} \, \Sigma^{\delta,r}_{a,a'}(\Phi (X) \,|\, \cdot\,) \Big]
  \d b \d r,
\end{split}
\end{equation}
where $\langle \cdot , \cdot \rangle$ denotes the $L^2$ inner product on $[0,1]$, see Theorem \ref{thm_ibpf_positive_cond} below. Here we used the abusive but convenient notation
\[ \Sigma^{\delta,r}_{a,a'}(\Phi(X) \,|\, b) := \int \Phi(X) \ \Sigma^{\delta,r}_{a,a'}({\rm d}X \,|\, b), \qquad b \geq 0.
\]
We stress that, for $\Phi$ as above, by Lemma \ref{lap_cond_bridge} below, the term
\[ \mathcal{T}^{\,2k}_{b} \, \Sigma^{\delta,r}_{a,a'}(\Phi (X) \,|\, \cdot\,) \]
appearing in the formulae is actually the Taylor remainder, centered at $0$, of a smooth function of $b^{2}$. In particular, it is of order $b^{2(k+1)}$ as $b \to 0$, which ensures the integral to be convergent. We also obtain the following formula for the critical case $\delta=1$  


\[
\begin{split}
& E^{1}_{a,a'} (\partial_{h} \Phi (X) ) + E^{1}_{a,a'} (\langle h '' , X \rangle \, \Phi(X) ) = \frac{1}{4} \int_{0}^{1} \d r \, h_r\, \frac{{\rm d}^{2}}{{\rm d} b^{2}} \, \Sigma^{1,r}_{a,a'} (\Phi(X) \, | \, b)  \biggr\rvert_{b=0}  .
\end{split}
\]
%
Note in particular that, in the case $a=a'=0$, we immediately recover the formulae of Theorem 4.1 of \cite{EladAltman2019}. 

The proof of the IbPFs for $P^{\delta}_{a,a'}$ ($a,a' \geq 0$) is a little more involved, although close in spirit to the particular case $a=a'=0$. Indeed, in the case $a=a'=0$, one could rely on the fact that quantities of the form $\Sigma^{\delta,r}_{0,0}(\Phi \,|\, b)$, for $\Phi$ of the form \eqref{exp_functional}, have a very simple expression: these are, up to some constants, just exponential functions in $b^{2}$, see \eqref{nice_expression} above and Lemma 3.6 of \cite{EladAltman2019}. Thus, in that case, the IbPFs all reduced (after some transformations) to an elementary identity on the $\Gamma$ function. Instead, in the case of general $a,a' \geq 0$, quantities of the form $\Sigma^{\delta,r}_{a,a'}(\Phi \,|\, b)$ for $\Phi \in S$  are more complicated functions in $b^2$, see Lemma \ref{lap_cond_bridge} below. Rather than merely relying on explicit computations - which, in the present case, would quickly become intractable - we are thus led to better understanding the structure of the right-hand side of the IbPFs. It turns out that, for any value of $\delta >0$, these can all be expressed using the family $(\mu_{\alpha})_{\alpha \in \mathbb{R}}$ of Schwartz distributions on $\mathbb{R}_{+}$  defined by \eqref{integer_alpha} and \eqref{non_integer_alpha} above, disregarding whether $\delta$ is larger or smaller than $3$. The only difference is the following: for $\alpha \geq 0$, corresponding to the case $\delta \geq 3$, $\mu_{\alpha}$ is a positive measure, while for $\alpha < 0$, corresponding to $\delta <3$, $\mu_{\alpha}$ is a genuine distribution. 
Although this difference is the source of a tremendous challenge in the study of the Bessel SPDEs associated with $\delta<3$, it does not really matter at the level of the IbPFs, which can all be derived by exploiting Lemma \ref{lap_cond_bridge_uncond} below, as well as elementary properties satisfied by $(\mu_\alpha)_{\alpha \in \mathbb{R}}$ (see Section \ref{sect_prelude} below) and the equation satisfied by the densities of the $\delta$-dimensional squared Bessel bridge,
see \eqref{pde_densities_only_adjoint} below.

\subsubsection{A dynamical result: the case $\delta=2$}

Our IbPFs allow us to construct a weak version of the dynamics associated with Bessel bridges of dimension $\delta = 2$ from $0$ to $0$ on $[0,1]$. We stress that the case $\delta=1$ has already been treated in \cite{EladAltman2019}. Here, using Dirichlet Form techniques, we  construct a Markov process $(u_t)_{t \geq 0}$ with the law of a $2$-dimensional Bessel bridge as reversible measure, and satisfying the SPDE

\begin{equation}\label{formal2}
\frac{\partial u}{\partial t}=\frac 12
\frac{\partial^2 u}{\partial x^2}
 - \frac{1}{8} \, \underset{\epsilon \to 0}{\lim} \,  \underset{\eta \to 0}{\lim} \, \left( \frac{\mathbf{1}_{\{u \geq \epsilon\}}}{u^{3}} - \frac{2}{\epsilon} \frac{\rho_{\eta}(u)}{u} \right) + \xi, \qquad (\delta=2),
 \end{equation}
where  $\rho_{\eta} = \frac{1}{\eta} \rho(\frac{x}{\eta})$ is a smooth approximation of the Dirac measure at $0$, see Theorem \ref{weak_spde_2}, for the precise statements. Heuristically, if one assumes that $u$ admits a family of local times $\ell^b_{t,x}, \, x \in (0,1), \, b,t \geq 0$ satisfying \eqref{occupation_time_formula} with $\delta=2$, then \eqref{formal2} is equivalent to \eqref{1<spde<3} for $\delta=2$. Thus, \eqref{formal2} is a weaker version of \eqref{1<spde<3} for $\delta=2$, in the sense that it does not a priori require the existence of local times. The techniques used in this article to obtain \eqref{formal2} are the same as those used in \citep{EladAltman2019} to obtain \eqref{formal1} for the case $\delta=1$, but the computations are slightly more involved. As in \cite{EladAltman2019}, the reason for focusing on integer values of $\delta$ (which, in the case $\delta<3$, only leaves $\delta=1,2$) consists in the existence of a  handy representation of the law of an integer-valued Bessel bridge in terms of the Euclidean norm of a Brownian bridge: since the gradient dynamics associated with a Brownian bridge are well-known and described by a linear Ornstein-Uhlenbeck process, in the integer-dimensional case, many problems thus boild down to relatively simple Gaussian computations. We stress that, prior to \citep{EladAltman2019}, this fact had already been exploited in several works to study the case $\delta=1$: in \cite{zambotti2005integration} and \cite{grothaus2016integration} for the derivation of IbPFs and in Section 5 of \cite{vosshallthesis} for the construction of the Markov process. Note that, even in the case of integer dimensions, such a Gaussian representation holds only in the case where one of the boundary values is $0$ (see \cite{yor2004remark}), so the method we use only applies to the cases $a=0$ or $a'=0$. In this article, for simplicity, the dynamics for $\delta=2$ will only be tackled in the case $a=a'=0$.

The article is organised as follows: 
in Section \ref{sect_sqred_bessel} we recall and prove some useful facts on the laws of squared Bessel processes, Bessel processes, and their bridges. In Section \ref{sect_ibpf}, we state and prove the IbPFs for the laws of Bessel bridges with arbitrary boundary values. 
Finally, the formulae for $\delta=2$ are used in Section \ref{sect_dynamics} to construct a weak form of the  corresponding SPDE, using Dirichlet Form techniques.

{\bf Acknowledgements.}
I am especially indebted to Lorenzo Zambotti for introducing me to this research topic as well as for countless precious discussions. The arguments used in Prop \ref{closability_2} below to show quasi-regularity of the forms associated with the law of a Bessel bridge of dimension $2$ were communicated to me by Rongchan Zhu and Xiangchan Zhu, whom I warmly thank. 

\section{Squared Bessel processes, Bessel processes, and associated bridges}
\label{sect_sqred_bessel}

In this section we briefly recall the definitions of squared Bessel processes, Bessel processes, and their corresponding bridges, as well as some useful facts. 

\subsection{An important family of distributions}
\label{sect_prelude}

We start by recalling the definition of a family of distributions on $\mathbb{R}_{+}$ already used in \cite{EladAltman2019}, and which can be seen as a simplified version of the laws of Bessel processes or bridges. More than a toy model, these objects will be an essential tool in the proof of the IbPFs below. We consider the Borel measures on $\mathbb{R}_{+}$, defned, for $\alpha > 0$, by
\begin{equation}
\label{def_mu_pos} 
\mu_{\alpha}(dx) = \frac{x^{\alpha - 1}}{\Gamma(\alpha)} \d x. 
\end{equation}
Moreover, we define $\mu_{0}$ as 
\[\mu_{0} = \delta_{0}, \] 
the Dirac measure at $0$. 
On the other hand, for $\alpha <0$, $\mu_\alpha$ is defined as a Schwartz distribution. We firts recall the appropriate space of test functions.

\begin{df}
Let $S([0,\infty))$ be the space of $C^\infty$ functions $\varphi: [0,\infty) \to \mathbb{R}$ such that, for all $k, l \geq 0$, there exists $C_{k,\ell} \geq0$ such that
\begin{equation}
\label{inequalities_schwartz}
\forall x \geq 0, \quad | \varphi^{(k)} (x)  | \, x^{\ell} \leq C_{k,\ell}.
\end{equation}
\end{df}
For each $\alpha \geq 0$, the measure $\mu_{\alpha}$ defines a Schwartz distribution (which we still denote by $\mu_{\alpha}$) as follows: for all test function $\varphi \in S([0,\infty))$,
\begin{equation}
\label{mellin_transform}
\langle \mu_{\alpha} , \varphi \rangle := \int_{0}^{\infty} \varphi(x) \d \mu_{\alpha}(x).
\end{equation}
Note that, due to the exponential decay of $\varphi$ at $\infty$, the above integral is indeed convergent. 
%
For any smooth function $f: \mathbb{R}_{+} \to \mathbb{R}$, for all $n \in \mathbb{Z}$, and all $x \geq 0$, we set
\[ \mathcal{T}^{\,n}_{x} f := f(x) - \sum_{0\leq j\leq n} \frac{x^{j}}{j!} \, f^{(j)}(0). \]
If $n<0$ then $\mathcal{T}^{\,n}_{x} f$ is simply the value of $f$ at $x$. With these notations, we recall the following definition from \cite{EladAltman2019}:
\begin{df}
\label{def_mu}
For $\alpha < 0$, we define the distribution  $\mu_{\alpha}$ as follows: 
\begin{itemize}
\item if $ \alpha = -k$ with $k \in \mathbb{N}$, then
\begin{equation}
\label{mu_neg_int} 
\langle \mu_{\alpha}, \varphi \rangle := (-1)^{k} \varphi^{(k)}(0), \qquad \varphi \in \mathcal{S}([0,\infty)) 
\end{equation}
\item if $ - k - 1 < \alpha < -k$ with $k \in \mathbb{N}$, then
\begin{equation} 
\label{mu_neg_delta}
\langle \mu_{\alpha} , \varphi \rangle := \int_{0}^{+ \infty} \mathcal{T}^{\,k}_{x} \varphi \, \frac{x^{\alpha -1}}{\Gamma(\alpha)} \d x, \qquad \varphi \in \mathcal{S}([0,\infty)).
\end{equation}

\end{itemize}
\end{df}


We stress that the above definition is very classical: for all $\alpha \in \mathbb{R}$, $\mu_\alpha$ coindicides with the generalised functional $\frac{x_+^{\alpha-1}}{\Gamma(\alpha)}$ of Section 3.5 of \cite{MR3469458}. In particular, for any fixed $\varphi \in \mathcal{S}([0,\infty))$, the function $\alpha \to \langle \mu_{\alpha}, \varphi \rangle$ is analytic on $\mathbb{R}$. We recall the following elementary formula (see (5) in Section 3.5 of \cite{MR3469458}), which states that $\mu_{\alpha-1}$ is the distributional derivative of $\mu_{\alpha}$. It can also be seen as a simplified version of the IbPFs in Section \ref{sect_ibpf} below.  

\begin{prop}
\label{thm_ibpf_mu}
The following formulae hold:
\[ \langle \mu_{\alpha}, \varphi' \rangle = - \langle \mu_{\alpha-1}, \varphi \rangle \]
for all $\alpha \in \mathbb{R}$ and $\varphi \in \mathcal{S}([0,\infty))$.
\end{prop}
As shown by Prop \ref{thm_ibpf_mu}, the family of distributions $\left( \mu_{\alpha} \right)_{\alpha \in \mathbb{R}}$ behaves nicely under differentiation. Actually it also behaves nicely under multiplication by $x$, as shown by the following result:

\begin{lm}
\label{lemma_mu}
For all $\alpha \in \mathbb{R}$ and $\varphi \in \mathcal{S}([0,\infty))$, the following relation holds:
\begin{equation}
\label{relation_multiplication}
\langle \mu_{\alpha} (x) , x \, \varphi(x) \rangle = \alpha \, \langle \mu_{\alpha+1}, \varphi \rangle.
\end{equation}
\end{lm}
Here we wrote a dummy variable $x$ to indicate which variable is being integrated, a convention we will also use below.

\begin{proof}
Let $\varphi \in \mathcal{S}([0,\infty))$. If $\alpha >0$, then \eqref{relation_multiplication} follows from the definition \eqref{def_mu_pos}. But since both sides of \eqref{relation_multiplication} are analytic in $\alpha$, the equality extends to any $\alpha \in \mathbb{R}$, and the claim follows.
\end{proof}

\subsection{Squared Bessel processes and Bessel processes} 

Here and below, for all $I \subset \mathbb{R}_+$, we shall denote by $C(I)$ the space of continuous, real-valued functions on $I$. For all $x, \delta \geq 0$, let $Q^{\delta}_{x}$ be the law, on $C(\mathbb{R}_{+})$, of a $\delta$-dimensional squared Bessel process started at $x$, which is defined as the solution to the SDE
\begin{equation} \label{sdeq} X_{t} = x + 2 \int_{0}^{t} \sqrt{X_{s}} \d B_{s} + \delta t, \qquad t\geq 0,
\end{equation}
where $B$ is a standard Brownian motion, see Chapter XI of \cite{revuz2013continuous}. 
%
We recall that squared Bessel processes are homogeneous Markov processes on $\mathbb{R}_{+}$, and that the transition densities $\left( q^{\delta}_{t}(x,y) \right)_{t > 0, x,y \geq 0}$ are explicitly known (see section XI of \cite{revuz2013continuous}). When $\delta >0$, these are given by
\begin{equation}
\label{density_besq_x_pos}
q^{\delta}_{t}(x,y) = \frac{1}{2t} \left( \frac{y}{x} \right)^{\nu/2} \exp\left( - \frac{x+y}{2t} \right) I_{\nu} \left(\frac{\sqrt{xy}}{t} \right),\quad t >0, \ x>0
\end{equation}
and
\begin{equation}
\label{density_besq_x_zero}
q^{\delta}_{t}(0,y) = (2t)^{-\frac\delta2} \, \Gamma \left( \delta/2 \right)^{-1} y^{\delta/2-1} \exp\left( - \frac{y}{2t} \right),\quad t >0.
\end{equation}
%
Above, $\nu := \delta/2 -1$ and $I_{\nu}$ is the modified Bessel function of index $\nu$:
\[ I_\nu(z) := \sum_{k=0}^\infty \frac{\left(z/2\right)^{2k + \nu}}{k! \, \Gamma(k + \nu +1)}, \qquad z > 0. \]
If $(X_{t})_{t \geq 0}$ is a $\delta$-dimensional squared Bessel process started at $x$, then the process $(\sqrt{X_{t}})_{\, t \geq 0}$ is, by definition, a $\delta$-dimensional Bessel process started at $a$, where $a=\sqrt{x}$. We shall denote by $\left( p^{\delta}_{t}(a,b) \right)_{t >0, \, a,b \geq 0}$ the corresponding transition densities. These are given in terms of the densities of the squared Bessel process by the relation

\begin{equation}
\label{relation_denisties_bes_besq}
 \forall t > 0, \forall a, b \geq 0, \quad p^{\delta}_{t}(a,b) = 2 \, b \, q^{\delta}_{t}(a^{2},b^{2}). 
\end{equation}
Note that the measure $\mu_{\delta}$ defined above is reversible for the $\delta$-dimensional Bessel process. Indeed, the following detailed balance condition holds:
\[ \forall t > 0, \forall a, b \geq 0, \quad a^{\delta-1} p^{\delta}_{t}(a,b) = b^{\delta-1} p^{\delta}_{t}(b,a). \]

\subsection{Squared Bessel bridges and Bessel bridges}


For all $\delta, x, y \geq 0$, we denote by $Q^{\delta}_{x,y}$ the law, on $C([0,1])$, of the $\delta$-dimensional squared Bessel bridge from $x$ to $y$ over the interval $[0,1]$, which is the law of a $\delta$-dimensional squared Bessel process started at $x$, and conditioned to hit $y$ at time $1$. A rigourous construction of these probability laws is provided in Chap. XI.3 of \cite{revuz2013continuous} (see also \cite{pitman1982decomposition} for a discussion on the particular case $\delta=y=0$). Note that, if $X \overset{(d)}{=} Q^{\delta}_{x,y}$, and $t \in (0,1)$, then the random variable $X_{t}$ admits the density $q^{\delta,t}_{x,y}$ on $\mathbb{R}_{+}$, where  
\begin{equation}
\label{one_pt_density_sqred_bridge} 
q^{\delta,t}_{x,y}(z) := \frac{q^{\delta}_{t} (x,z) q^{\delta}_{1 -t} (z, y)}{ q^{\delta}_{1}(x,y)},\quad z \geq 0 
\end{equation} 
see Chap. XI.3 of \cite{revuz2013continuous}. Note the following continuity property: for all $\delta >0$, the map $(x,y) \mapsto Q^{\delta}_{x,y}$ is continuous on $\mathbb{R}_{+}^{2}$ for the weak topology on probability measures (see Chap. XI.3 in \cite{revuz2013continuous}).
%
%

In the sequel, for any $\delta, a, a' \geq 0$, we shall denote by $P^{\delta}_{a,a'}$ the law, on $C([0,1])$, of the $\delta$-dimensional Bessel bridge from $a$ to $a'$ over the time interval $[0,1]$, that is the law of a $\delta$-dimensional Bessel process started at $a$ and conditioned to hit $a'$ at time $1$. We shall denote by $E^{\delta}_{a,a'}$ the associated expectation operator. Note that $P^{\delta}_{a,a'}$ is the image of $Q^{\delta}_{a^{2},a'^{2}}$ under the map 
\begin{align}
\label{sqrt_map}
C([0,1])\ni \omega \mapsto \mathbf{1}_{\omega \geq 0} \, \sqrt{\omega} \in C([0,1])  .
\end{align}
In particular, if $X \overset{(d)}{=} P^{\delta}_{a,a'}$, and $r \in (0,1)$, then $X_{r}$ admits the density $p^{\delta,r}_{a,a'}$ on $\mathbb{R}_{+}$, where for all $a \geq 0$ and $a'>0$,
\begin{equation}\label{one_pt_density_bridge}
p^{\delta,r}_{a,a'}(b) = \frac{p^{\delta}_{r}(a,b)p^{\delta}_{1-r}(b,a')}{p^{\delta}_{1}(a,a')}, \quad b \geq 0, 
\end{equation} 
see \cite[Chapter XI.3]{revuz2013continuous}.
In the case $a'=0$, the corresponding density is given by
\[p^{\delta,r}_{a,0}(b) := \underset{a' \to 0}{\lim} \, p^{\delta,r}_{a,a'}(b), \quad b \geq 0,\] 
see Remark \ref{analytic_ext} below. In the particular case $a=a'=0$, consistently with the notations used in \cite{EladAltman2019}, we shall write $P^{\delta}$ instead of $P^\delta_{0,0}$, and $p^{\delta}_{r}(b)$ instead of $p^{\delta,r}_{0,0}(b)$, for short. Recall that the following formula then holds:
\[p^{\delta}_{r}(b) := \frac{b^{\delta-1}}{2^{\frac\delta2-1}\,\Gamma(\frac{\delta}{2})(r(1-r))^{\delta/2}}\, \exp \left(- \frac{b^{2}}{2r(1-r)} \right), \quad b \geq 0. \]
We finally introduce the last family of measures that we shall manipulate, which are further conditioned versions of the stochastic processes considered above. 

\subsection{Pinned bridges}

For all $\delta, x,y, z \geq 0$ and $r \in (0,1)$, we will denote by $Q^{\delta}_{x,y} [\ \cdot \ \, | \, X_{r} = z]$ the law, on $C([0,1])$, of a $\delta$-dimensional squared Bessel bridge between $x$ and $y$, pinned at $z$ at time $r$, that is the law of a $Q^{\delta}_{x,y}$ bridge conditioned to hit $z$ at time $r$. Such a probability law can be constructed using the same conditioning procedure as for the construction of squared Bessel bridges. Similarly one also considers, for all $\delta, a,b,c \geq 0$ and $r \in (0,1)$, the law  $P^{\delta}_{a,b} [\ \cdot \ \, | \, X_{r} = c]$ of a $\delta$-dimensional Bessel bridge between $a$ and $b$, pinned at $c$ at time $r$. Note that this probability measure is then the image of $Q^{\delta}_{a^{2},b^{2}} [ \ \cdot \ \, | \, X_{r} = c^{2}]$ under the map \eqref{sqrt_map}. With these notations at hand, we now define a family of measures generalising Definition 3.4 of \cite{EladAltman2019} to the setting of bridges with general boundary values. The idea motivating this definition is the same as in the case of vanishing boundary values: these measures should be the Revuz measures of the local time processes of the solution $(u(t,x))_{t\geq 0, \, x \in [0,1]}$ to an SPDE with reversible measure given by $P^{\delta}_{a,a'}$, for $a,a' \geq 0$.

\begin{df}
For all $a,a', b \geq 0$ and $r\in(0,1)$, we set
\begin{equation}\label{Sigma}
\Sigma^{\delta,r}_{a,a'}({\rm d}X \,|\, b) := \frac{p^{\delta,r}_{a,a'}(b)}{b^{\delta-1}} \,
 P^{\delta}_{a,a'} [ \d X | \, X_{r} = b],
\end{equation}
where $p^{\delta,r}_{a,a'}$ is the probability density function of $X_{r}$ under $P^{\delta}_{a,a'}$, see \eqref{one_pt_density_bridge}. 
\end{df}
As mentioned above, the measure $\Sigma^{\delta,r}_{a,a'}(\,\cdot \,|\, b)$ is meant to be the Revuz measure of the diffusion local time of $(u(t,r))_{t\geq 0}$ at level $b\geq 0$. Note in particular that, for $a=a'=0$, $\Sigma^{\delta,r}_{0,0}({\rm d}X \,|\, b)$ coincides with the measure $\Sigma^{\delta}_{r}({\rm d}X \,|\, b)$ introduced in Def 3.4 of \cite{EladAltman2019}. For the sake of concision, for all $r \in (0,1)$ and $a,a',b \geq 0$, and all Borel function $\Phi : C([0,1]) \to \mathbb{R}_+$, we  write with a slight abuse of notation
\[ \Sigma^{\delta,r}_{a,a'}(\Phi(X) \,|\, b) := \int \Phi(X) \ \Sigma^{\delta,r}_{a,a'}({\rm d}X \,|\, b), \qquad b \geq 0.
\]
\begin{rk}
\label{analytic_ext}
The equalities \eqref{Sigma} and\eqref{one_pt_density_bridge} above \textit{do} also include the cases $b=0$ and $a'=0$. Indeed, note that, as a consequence of the expressions \eqref{density_besq_x_pos}, \eqref{density_besq_x_zero} and \eqref{relation_denisties_bes_besq}, $\frac{p^{\delta}_{t}(a,b)}{b^{\delta-1}}$ can be extended to a smooth (actually analytic) function of $b$, at $b=0$. Similarly, for all $a, b \geq 0$, the function
\[a' \to \frac{p^{\delta}_{1-r}(b,a')}{p^{\delta}_{1}(a,a')}\]
can be extended in an analytic way at $a'=0$. In the sequel, we will systematically consider these analytic extensions.  
\end{rk}
Let $m$ be a finite Borel measure on $[0,1]$. In the sequel we will have to compute quantities of the form
\[
\Sigma^{\delta,r}_{a,a'}\left[\exp(- \langle m, \, X^2 \rangle) \,|\, b\right], 
\]
where we use the shorthand notation $ \langle m, \, X^2 \rangle := \int_0^1 X_t^2 \, m({\rm{d}} t)$. As in Chap. XI of \cite{revuz2013continuous} and Section 3 of \cite{EladAltman2019}, we consider $\phi=(\phi_r,r\geq 0)$ the unique solution, on $\mathbb{R}_{+}$, of the following problem:
\begin{equation}
\label{phi} \tag{$SL_{m}$}
\begin{cases}
\phi''(\d r) = 2 \, \mathbf{1}_{[0,1]}(r) \, \phi_{r} \, m({\rm{d}} r)  \\
\phi_0=1, \\
\phi > 0,  \phi ' \leq 0  \ \text{on} \ \mathbb{R}_{+},
\end{cases}
\end{equation} 
where the first equality is in the sense of distributions (see Appendix 8 of \cite{revuz2013continuous} for existence and uniqueness of a solution to this problem). As in Section 3 of \cite{EladAltman2019}, we also set
\begin{equation}
\label{definition_varrho}
\varrho_r :=   \int_{0}^{r} \phi_u^{-2} \d u, \qquad r \in [0,1].\end{equation}
With these notations at hand, we obtain the following result, which is a generalisation of Lemma 3.6 in \cite{EladAltman2019}:

\begin{lm}\label{lap_cond_bridge} 
For all $r \in (0,1)$, $\delta>0$ and $a,a',b \geq 0$, the following holds:
\begin{equation}\label{general_bridg2}
\begin{split}
& \int \exp (- \langle m, X^{2} \rangle) \ \Sigma^{\delta,r}_{a,a'}({\rm d}X \,|\, b) \\
& = 2 \exp \left( \frac{a^{2}}{2} \phi'_{0} \right)  \phi_1^{\delta/2-2} \phi_r^{-2} 
\frac{q^{\delta}_{\varrho_r} \left(a^{2},\frac{b^{2}}{\phi_r^{2}}\right) q^{\delta}_{\varrho_1 - \varrho_r}\left(\frac{b^{2}}{\phi_r^{2}},\frac{a'^{2}}{\phi_1^{2}}\right)}{b^{\delta-2} \, q^{\delta}_{1}(a^{2},a'^{2})} \\
& = 2 \exp \left( \frac{a^{2}}{2} \phi'_{0} \right)  \phi_1^{\delta/2-2} \phi_r^{-\delta} 
\varrho_{1}^{-\delta-1} \frac{q^{\delta,t}_{x,y} (z)}{z^{\delta/2-1}},
\end{split}
\end{equation}
where $x=\frac{a^{2}}{\varrho_{1}}$, $y=\frac{a'^{2}}{\varrho_{1} \phi_{1}}$, $z=\frac{b^{2}}{\varrho_{1} \phi_{r}^{2}}$, and $t=\frac{\varrho_{r}}{\varrho_{1}} \in [0,1]$. Here, $q^{\delta,t}_{x,y}$ denotes the density of the random variable $X_{t}$, when $X \overset{(d)}{=} Q^{\delta}_{x,y}$, see \eqref{one_pt_density_sqred_bridge}. 
\end{lm} 


\begin{rk}
The above lemma shows that, for all measure $m$ as above, all $a,a' \geq 0$ and $r \in (0,1)$, the function
\[b \to \int \exp (- \langle m, X^{2} \rangle) \ \Sigma^{\delta,r}_{a,a'}({\rm d}X \,|\, b) \]
is a smooth (actually analytic) function of $b^2$. In particular
\begin{equation}
\label{vanishing_derivative_sigma}
\frac{\rm d}{{\rm d}b} \left( \,\int \exp (- \langle m, X^{2} \rangle) \ \Sigma^{\delta,r}_{a,a'}({\rm d}X \,|\, b) \right) \biggr\rvert_{b=0} = 0. 
\end{equation}
\end{rk}


\begin{rk}
As a consequence of \eqref{general_bridg2}, in the special case $a=a'=0$, recalling the expressions \eqref{density_besq_x_pos} and \eqref{density_besq_x_zero}, we obtain
\begin{equation*}
\begin{split}
&\int \exp (- \langle m, X^{2} \rangle) \ \Sigma^{\delta,r}_{0,0}({\rm d}X \,|\, b) =\\
&= \frac1{2^{\frac\delta2-1}\,\Gamma(\frac{\delta}{2})} \,
\exp \left(-\frac{b^{2}\varrho_{1}}{2 \phi_r^{2}  \varrho_r (\varrho_1 - \varrho_r)} \right)   \left(2 \phi_r^{2} \phi_{1} \varrho_r (\varrho_1 - \varrho_r) \right)^{-\delta/2},
\end{split}
\end{equation*} 
which coincides with the formula (3.15) in \cite{EladAltman2019}.
\end{rk}

\begin{proof}[Proof of Lemma \ref{lap_cond_bridge}]
First note that by the relation \eqref{relation_denisties_bes_besq} and by the expression \eqref{Sigma}, we have

\begin{equation}
\label{intermediate_expr} 
\begin{split}
& \int \exp (- \langle m, X^{2} \rangle) \ \Sigma^{\delta,r}_{a,a'}({\rm d}X \,|\, b) = \\
& = 2  \frac{q^{\delta}_{r}(a^{2},b^{2}) q^{\delta}_{1-r}(b^{2},a'^{2})}{b^{\delta-2} \, q^{\delta}_{1}(a^{2},a'^{2})}  Q^{\delta}_{a^{2},a'^{2}} [\exp (- \langle m, X \rangle ) \, | \, X_{r} = b^{2}]. 
\end{split}
\end{equation}
To obtain the claim, it therefore suffices to compute 
\[ Q^{\delta}_{a^{2},a'^{2}} [\exp (- \langle m, X \rangle ) \, | \, X_{r} = b^{2}], \]
a quantity which we can rewrite as
\[Q^{\delta}_{a^{2}} [\exp (- \langle m, X \rangle) \, | \, X_{r} = b^2, X_{1}=a'^2]. \]
But, arguing as in the proof of Lemma 3.6 in \cite{EladAltman2019}, we deduce from Lemma 3.3 in \cite{EladAltman2019} that, for all $x,y \geq 0$  
\[ \begin{split} 
& Q^{\delta}_{a^{2}} [\exp (- \langle m, X \rangle) \, | \, X_{r} = x, X_{1}=y] = \\
& = \exp \left( \frac{a^{2}}{2} \phi'_{0} \right)  \phi_1^{\delta/2-2} \phi_r^{-2} \frac{q^{\delta}_{\varrho_r} \left(a^{2},\frac{x}{\phi_r^{2}}\right) q^{\delta}_{\varrho_1 - \varrho_r}\left(\frac{x}{\phi_r^{2}},\frac{y}{\phi_1^{2}}\right)}{q^{\delta}_{r}(a^{2},x) \,q^{\delta}_{1-r}(x,y)}. 
\end{split}\]
Applying this equality to $x=b^{2}$ and $y=a'^{2}$, and replacing in \eqref{intermediate_expr}, we obtain \eqref{general_bridg2}


\end{proof}

\section{Integration by parts formulae}

\label{sect_ibpf}

As in \cite{EladAltman2019}, we denote by $S$ the linear span of the set of functionals on $C([0,1])$ of the form
\begin{align}
\begin{cases}
C([0,1]) \to \mathbb{R} \\
X \mapsto \exp \left( - \langle m, X^{2} \rangle \right),
\end{cases}
\end{align}
where $m$ is a finite Borel measure on $[0,1]$. The elements of $S$ are the functionals for which we will derive our IbPFs with respect to the laws of Bessel bridges. In the sequel, we shall also use the notation $\langle \cdot, \cdot \rangle$ for the $L^2$ inner product on $(0,1)$:
\[ \langle f, g \rangle := \int_0^1 f_r \, g_r \, \d r, \quad f,g \in L^2(0,1). \]

\subsection{The statement}

Recalling the definition
\[\kappa(\delta) := \frac{(\delta-3)(\delta-1)}{4}, \qquad \delta\in\R, \]
we can now state the main theorem of this section, which generalises Theorem 4.1 of \cite{EladAltman2019} to the case of Bessel bridges with arbitrary boundary values:

\begin{thm}
\label{thm_ibpf_positive_cond}
Let $a,a' \geq 0$, $\delta \in (0,\infty) \setminus \{1,3\}$, and $k:=\lfloor \frac{3-\delta}{2} \rfloor \leq 1$. Then, for all $\Phi \in S$ and $h \in C^2_c(0,1)$,
\begin{equation}
\label{exp_fst_part_ibpf_a_b}
\begin{split}
& E^{\delta}_{a,a'} (\partial_{h} \Phi (X) ) + E^{\delta}_{a,a'} (\langle h '' , X \rangle \, \Phi(X) )= 
 \\ & =-\kappa(\delta)\int_{0}^{1}  
 h_{r}  \int_0^\infty b^{\delta-4} \Big[ \mathcal{T}^{\,2k}_{b} \, \Sigma^{\delta,r}_{a,a'}(\Phi (X) \,|\, \cdot\,) \Big]
  \d b \d r.
\end{split}
\end{equation}
On the other hand, when $\delta \in \{1,3\}$, the following formulae hold: for all $\Phi \in S$ and $h \in C^2_c(0,1)$,
\begin{equation}
\label{exp_fst_part_ibpf_a_b_3}
E^{3}_{a,a'} (\partial_{h} \Phi (X) ) + E^{3}_{a,a'} (\langle h '' , X \rangle \, \Phi(X) ) = 
-\frac{1}{2} \int_{0}^{1} h_{r} \, \Sigma^{3,r}_{a,a'}(\Phi (X) \,|\, 0\,) \d r, 
\end{equation}
and
\begin{equation}
\label{exp_fst_part_ibpf_a_b_1}
\begin{split}
E^{1}_{a,a'} (\partial_{h} \Phi (X) ) + E^{1}_{a,a'} (\langle h '' , X \rangle \, \Phi(X) ) = 
\frac{1}{4} \int_{0}^{1} h_{r} \, \frac{{\rm d}^{2}}{{\rm d} b^{2}} \, \Sigma^{1,r}_{a,a'} (\Phi(X) \, | \, b)  \biggr\rvert_{b=0}  \d r. 
\end{split}
\end{equation}
\end{thm}

%

\begin{rk}
For all $\delta \in (1,3)$ the right-hand side in the IbPF \eqref{exp_fst_part_ibpf_a_b} takes the form:
\[ -\kappa(\delta)\int_{0}^{1} h_{r}  \int_0^\infty b^{\delta-4} \Big[ \Sigma^{\delta,r}_{a,a'}(\Phi (X) \,|\, b) - \Sigma^{\delta,r}_{a,a'}(\Phi (X) \,|\, 0) \Big]  \d b \d r. \]
Thus, as already noted in Remark 4.3 of \cite{EladAltman2019} for the case of bridges from $0$ to $0$, there is no transition at $\delta=2$ at the level of the IbPFs. 
However, we do conjecture that a transition should occur for $\delta=2$ at the level of the SPDEs: see Section \ref{sect_dynamics} below.
\end{rk}

\begin{rk}
Recalling the definition \ref{def_mu} of $\mu_{\alpha}$ for $\alpha < 0$, we can write all the above IbPFs in a unified way as follows:
\begin{equation}
\label{unified_ibpf}
\begin{split}
& E^{\delta}_{a,a'} (\partial_{h} \Phi (X) ) + E^{\delta}_{a,a'} (\langle h '' , X \rangle \, \Phi(X) ) 
 \\ & =-\frac{\Gamma (\delta)}{4(\delta-2)} \int_{0}^{1}  
 h_{r} \, \langle \mu_{\delta-3}, \Sigma^{\delta,r}_{a,a'}(\Phi (X) \,|\, \cdot\,) \rangle \d r, 
\end{split}
\end{equation}
where the singularity at $\delta=2$ is compensated by the vanishing at $\delta=2$ of the quantity $\langle \mu_{\delta-3}, \Sigma^{\delta,r}_{a,a'}(\Phi (X) \,|\, \cdot\,) \rangle$ as a consequence of \eqref{vanishing_derivative_sigma}. 
Actually the proof of the formulae of Theorem \ref{thm_ibpf_positive_cond} will be based on rewriting both sides of the equalities using the family of distributions $(\mu_{\alpha})_{\alpha \in \mathbb{R}}$: see Lemma \ref{unconstrained_secondder} and its proof. Note that in that lemma there appears $\mu_{\frac{\delta-3}{2}}$ rather than $\mu_{\delta-3}$ because, for convenience, we work there with squared Bessel processes rather than Bessel processes.  
\end{rk}

%
As a consequence of the above theorem, we retrieve the following known results (see Chapter 6 of \cite{zambotti2017random}):
  
\begin{prop}
\label{already_known_ibpf0}
Let $\Phi \in S$ and $h \in C^{2}_{c}(0,1)$. Then, for all $a \geq 0$ and $\delta > 3$, the following IbPF holds:
\begin{equation}
\label{ibpf_3+0}
 E^{\delta}_{a,a} (\partial_{h} \Phi (X) ) + E^{\delta}_{a,a} (\langle h '' , X \rangle \, \Phi(X) ) = - \kappa(\delta) \, E^{\delta}_{a,a} (\langle h , X^{-3} \rangle \, \Phi(X) ). 
\end{equation}
Moreover, for $\delta = 3$, the following IbPF holds:
\begin{equation}
\label{ibpf_30}
\begin{split}
& E^{3}_{a,a} (\partial_{h} \Phi (X) )+ E^{3}_{a,a} (\langle h '' , X \rangle \, \Phi(X) ) = 
\\ & = - \int_{0}^{1} \d r \,  h_r \, \gamma(r,a) \, E^{3}_{a,a} [\Phi(X) \, | \, X_{r} = 0]
\end{split} 
\end{equation}
where, for all $(r,a) \in (0,1) \times \mathbb{R}_{+}$
\[ \gamma(r,a) := \frac{1}{\sqrt{2 \pi r^{3} (1-r)^{3}}} \left( \mathbf{1}_{a=0} +  \mathbf{1}_{a>0} \frac{2a^{2}\exp\left(-\frac{a^{2}}{2r(1-r)}\right)} {1-e^{-2a^{2}}}\right).  \] 

\end{prop}

\begin{proof}
For $\delta>3$, the result follows as in the proof of Prop. 4.5 of \cite{EladAltman2019}. 
For $\delta=3$, it suffices to note that, for all $r \in (0,1)$

\[ \frac{1}{2} \, \underset{\epsilon \to 0}{\lim} \,  \frac{p^{3,r}_{a,a}(\epsilon)}{\epsilon^{2}} = \gamma(r,a),  \]

so that

\[\frac{1}{2} \, \Sigma^{3,r}_{a,a}(\Phi (X) \,|\, 0\,)  = \gamma(r,a) E^{3}_{a,a} [\Phi(X) \, | \, X_{r} = 0]. \]
and the proof is complete.
\end{proof}

The next two sections are devoted to the proof of the above IbPFs. We will actually first state and prove similar IbPFs for the laws of Bessel processes (with the value of $X_1$ unconstrained) for which the computations are  lighter than in the case of bridges, and we will then obtain the results for Bessel bridges by conditioning.

\subsection{Case of unconstrained Bessel processes}

We first introduce the following: 

\begin{df}
For all $a, b \geq 0$ and $r\in(0,1)$, we consider the measure $\Sigma^{\delta,r}_{a}({\rm d}X \,|\, b)$ on $C([0,1])$ defined by
\begin{equation}\label{Sigma_unconst}
\Sigma^{\delta,r}_{a}({\rm d}X \,|\, b) := \frac{p^{\delta}_{r}(a,b)}{b^{\delta-1}} \,
 P^{\delta}_{a} [ \d X \, | \, X_{r} = b].
\end{equation}
\end{df}

\begin{lm}\label{lap_cond_bridge_uncond} 
For all $r \in (0,1)$, $\delta>0$ and $a,b \geq 0$, the following holds
\begin{equation}\label{general_bridg2_uncond}
\begin{split}
 \int \exp (- \langle m, X^{2} \rangle) \ \Sigma^{\delta,r}_{a}({\rm d}X \,|\, b) = 
 2 \exp \left( \frac{a^{2}}{2} \phi'_{0} \right)  \phi_1^{\delta/2} \phi_r^{-2} \frac{q^{\delta}_{\varrho_r} \left(a^{2},\frac{b^{2}}{\phi_r^{2}}\right)}{b^{\delta-2}},
\end{split}
\end{equation}
where $\phi$ and $\varrho$ are defined by \eqref{phi} and \eqref{definition_varrho}. In particular, for $a=0$, we have
\begin{equation*}
\int \exp (- \langle m, X^{2} \rangle) \ \Sigma^{\delta,r}_{0}({\rm d}X \,|\, b) =
\frac1{2^{\frac\delta2-1}\,\Gamma(\frac{\delta}{2})} \,
\exp \left(-\frac{b^{2}}{2\phi_{r}^{2} \varrho_{r}} \right)  \left( \frac{\phi_{1}}{\phi_{r}^{2} \varrho_{r}} \right)^{\delta/2}.
\end{equation*}

\end{lm} 

\begin{proof}
These equalities follow from Lemma \ref{lap_cond_bridge} upon noticing that, for all $a \geq 0$,

\[ \Sigma^{\delta,r}_{a}({\rm d}X \,|\, b) = \int_{0}^{\infty} \Sigma^{\delta,r}_{a,a'}({\rm d}X \,|\, b)  \, p^{\delta}_{1}(a,a') \, \d a'. \]

\end{proof}

\begin{thm}
\label{thm_ibpf_positive_cond_unconst}
Let $a\geq 0$, $\delta \in (0,\infty) \setminus \{1,3\}$, and $k:=\lfloor \frac{3-\delta}{2} \rfloor \leq 1$. Then, for all $\Phi \in S$ and $h \in C^2_c(0,1)$,
\begin{equation}
\label{exp_fst_part_ibpf_a_unconst}
\begin{split}
& E^{\delta}_{a} (\partial_{h} \Phi (X) ) + E^{\delta}_{a} (\langle h '' , X \rangle \, \Phi(X) )= 
 \\ &-\kappa(\delta)\int_{0}^{1}  
 h_{r}  \int_0^\infty b^{\delta-4} \Big[ \mathcal{T}^{\,2k}_{b} \, \Sigma^{\delta,r}_{a}(\Phi (X) \,|\, \cdot\,) \Big]
  \d b \d r.
\end{split}
\end{equation}
On the other hand, when $\delta \in \{1,3\}$, the following formulae hold for all $\Phi \in S$ and $h \in C^2_c(0,1)$:
\begin{equation}
\label{exp_fst_part_ibpf_a_3_unconst}
E^{3}_{a} (\partial_{h} \Phi (X) ) + E^{3}_{a} (\langle h '' , X \rangle \, \Phi(X) ) = 
-\frac{1}{2} \int_{0}^{1} h_{r} \, \Sigma^{3,r}_{a}(\Phi (X) \,|\, 0\,) \d r, 
\end{equation}
and
\begin{equation}
\label{exp_fst_part_ibpf_a_1_unconst}
\begin{split}
E^{1}_{a} (\partial_{h} \Phi (X) ) + E^{1}_{a} (\langle h '' , X \rangle \, \Phi(X) ) = 
\frac{1}{4} \int_{0}^{1} h_{r} \, \frac{{\rm d}^{2}}{{\rm d} b^{2}} \, \Sigma^{1,r}_{a} (\Phi(X) \, | \, b)  \biggr\rvert_{b=0}  \d r. 
\end{split}
\end{equation}
\end{thm}

\begin{proof}
By linearity, it suffices to prove the formulae \eqref{exp_fst_part_ibpf_a_unconst},\eqref{exp_fst_part_ibpf_a_3_unconst} and \eqref{exp_fst_part_ibpf_a_1_unconst} for $\Phi$ of the form $\eqref{exp_functional}$. So let $m$ be a finite Borel measure on $[0,1]$, and let $\Phi$ be the functional thereto associated. We start by computing the left-hand side of the above claimed formulae. We have
\begin{align*}
& E^{\delta}_{a} (\partial_{h} \Phi (X) ) + E^{\delta}_{a} (\langle h '' , X \rangle \, \Phi(X) ) = 
E^{\delta}_{a} \left( \langle h '' - 2 h m, X \rangle  \Phi(X) \right) \\
&=  \int_{0}^{1} \left( \d r \, h_{r}  \frac{ \d^{2}}{\d r^{2}} - 2 m(\d r) \, h_r \right) E^{\delta}_{a} [X_{r} \exp \left( - \langle m, X^{2} \rangle \right) ].
\end{align*}
We now claim that, for all $r \in (0,1)$,
\begin{equation}\label{thusobtain_unconst}
 E^{\delta}_{a} [X_{r} \exp( - \langle m, X^{2} \rangle) ] = \exp \left( \frac{a^{2}}{2} \phi'(0) \right) \phi_{1}^{\delta/2} \, \phi_r \, E^{\delta}_{a} \left( {X_{\varrho_r}} \right),
\end{equation}
where $\varrho$ is defined by \eqref{definition_varrho}. Indeed, we have
\begin{align*}
E^{\delta}_{a} [X_{r} \exp( - \langle m, X^{2} \rangle) ] = Q^{\delta}_{a^{2}} \left[\sqrt{X_{r}} \exp( - \langle m, X \rangle) \right].
\end{align*}
But, by Lemma 3.3 of \cite{EladAltman2019}, the quantity in the right-hand side equals
\[\exp \left( \frac{a^{2}}{2} \phi'(0) \right) \phi_{1}^{\delta/2} \, \phi_r \, Q^{\delta}_{a^{2}} \left( \sqrt{X_{\varrho_r}} \right), \]
and equality \eqref{thusobtain_unconst} follows. To alleviate notations, we rewrite \eqref{thusobtain_unconst} as follows:
\[E^{\delta}_{a} [X_{r} \exp( - \langle m, X^{2} \rangle) ] = K(a, m) \, \phi_r \, \zeta_{\varrho_r},\]
 where 
\[
K(a,m):=\exp \left( \frac{a^{2}}{2} \phi'(0) \right) \phi_{1}^{\delta/2} 
\] 
is a constant which does not depend on $r$ and
\[
\zeta_t := Q^{\delta}_{a^{2}} \left( \sqrt{X_t} \right) = E^{\delta}_{a}  \left(X_t \right), \quad t \geq 0.
\]
To compute the left-hand sides of \eqref{exp_fst_part_ibpf_a_unconst},\eqref{exp_fst_part_ibpf_a_3_unconst} and \eqref{exp_fst_part_ibpf_a_1_unconst}, it therefore suffices to compute the following distribution on $(0,1)$:
\[ \left( \frac{ \d^{2}}{\d r^{2}} - 2 \, m({\rm{d}} r) \right) \left(\phi_r  \zeta_{\varrho_r} \right). \]  
We recall that by \eqref{definition_varrho}
\[ \varrho'_{r}= \phi_r^{-2}. \]
By the Leibniz formula, we obtain
\[ \begin{split}
\frac{ \rm d}{{\rm d} r} \left(\phi_r \zeta_{\varrho_r}\right) = \phi_r' \zeta_{\varrho_r} + \phi_r  \frac{\zeta_{\varrho_r}'}{ \phi_r^{2}} = 
\phi_r' \zeta_{\varrho_r} +  \frac{\zeta_{\varrho_r}'}{\phi_r},
\end{split}\]
\[ \begin{split}
\frac{ \d^{2}}{\d r^{2}} \left(\phi_r \zeta_{\varrho_r}\right) = \phi_r'' \zeta_{\varrho_r} +
\phi_r'  \frac{\zeta_{\varrho_r}'}{\phi_r^2}- 
\phi_r'  \frac{\zeta_{\varrho_r}'}{\phi_r^2} + 
\frac{\zeta_{\varrho_r}''}{\phi_r^3} = \phi_r'' \zeta_{\varrho_r} +
\frac{\zeta_{\varrho_r}''}{\phi_r^3}.
\end{split}\]
Consequently, recalling that $\phi'' = 2 \phi \, m$, we obtain 
\[ 
\left( \frac{ \d^{2}}{\d r^{2}} - 2 \, m({\rm{d}} r) \right) \left(\phi_r \zeta_{\varrho_r} \right) = \frac{\zeta_{\varrho_r}''}{\phi_r^3}. 
\]
Finally, we thus obtain the following expression for the left-hand sides of \eqref{exp_fst_part_ibpf_a_unconst},  \eqref{exp_fst_part_ibpf_a_3_unconst} and \eqref{exp_fst_part_ibpf_a_1_unconst}:
\begin{equation}
\label{left_hand_side_ibpf_unconst}
 \begin{split}
& E^{\delta}_{a} (\partial_{h} \Phi (X) ) + E^{\delta}_{a} (\langle h '' , X \rangle \, \Phi(X) ) \\
& = K(a,m) \, 
\int_{0}^{1} \d r \, h_{r} \, \phi_r^{-3} \, \frac{\d^{2}}{\d t^{2}} E^{\delta}_{a}  \left( {X_{t}} \right) \biggr \rvert_{t=\varrho_r}.  
\end{split}
\end{equation}
We now compute the right-hand  sides of \eqref{exp_fst_part_ibpf_a_unconst},  \eqref{exp_fst_part_ibpf_a_3_unconst} and \eqref{exp_fst_part_ibpf_a_1_unconst}. Recall that, by \eqref{general_bridg2_uncond}, we have for all $r \in (0,1)$ and $b \geq 0$
\begin{equation*}
\begin{split}
\Sigma^{\delta,r}_{a}(\Phi (X) \,|\, b\,) &=
 2 \exp \left( \frac{a^{2}}{2} \phi'_{0} \right)  \phi_1^{\delta/2} \phi_r^{-2} \frac{q^{\delta}_{\varrho_r} \left(a^{2},\frac{b^{2}}{\phi_r^{2}}\right)}{b^{\delta-2}} \\
&= 2 K(a,m) \, \phi_{r}^{-2} \,   \frac{q^{\delta}_{\varrho_r} \left(a^{2},\frac{b^{2}}{\phi_r^{2}}\right)}{b^{\delta-2}}.
\end{split}
\end{equation*}
Therefore, setting $t:= \varrho_r$ and denoting by $f$ the function defined by
\[ 
f(y) :=   \frac{q^{\delta}_{t}(a^{2},y)}{y^{\delta/2-1}}, \quad y > 0, 
\]
and extended by continuity at $y=0$, we have
\begin{equation}
\label{formula_sigma_ab_unconst}
\Sigma^{\delta,r}_{a}(\Phi (X) \,|\, b\,) = 2 K(a, m) \, \phi_r^{-\delta} \, f \left(\frac{b^{2}}{ \phi_r^{2}} \right),  
\end{equation}
for all $b \geq 0$. Now, we first assume that $\delta \notin \{1,3\}$, and compute the right-hand side of \eqref{exp_fst_part_ibpf_a_unconst}. Note that, by \eqref{formula_sigma_ab_unconst}, and performing the change of variable $y:= \frac{b^{2}}{\phi_r^{2}}$, we obtain
\begin{equation}
\label{integral_unconst}
\begin{split}
\int_0^\infty \d b \ b^{\delta-4} \Big[ \mathcal{T}^{\,-2k}_{b} \, \Sigma^{\delta,r}_{a}(\Phi (X) \,|\, \cdot\,) \Big] = K(a,m) \, \phi_r^{-3} \int_0^\infty \d y \ y^{\frac{\delta-3}{2}-1} \, \mathcal{T}^{\,-k}_{y} f,
\end{split}
\end{equation}
where we recall that $k:=\lfloor \frac{3-\delta}{2} \rfloor \leq 1$. Recalling also the definition of $\mu_{\frac{\delta-3}{2}}$, we can rewrite the last integral of \eqref{integral_unconst} as
\begin{equation*}
 \Gamma \left( \frac{\delta-3}{2} \right) \left\langle \mu_{\frac{\delta-3}{2}} (y), f(y) \right\rangle.
\end{equation*}
Since $\Gamma \left( \frac{\delta+1}{2} \right) = \kappa(\delta) \, \Gamma \left( \frac{\delta-3}{2} \right)$, we thus deduce that the integrand in the right-hand side of \eqref{exp_fst_part_ibpf_a_unconst} equals
\begin{equation}
\label{right_hand_side_ibpf_unconst}
- K(a, m) \, \Gamma \left( \frac{\delta+1}{2} \right) \, h_r \, \phi_r^{-3} \, \left\langle \mu_{\frac{\delta-3}{2}} (y), f(y) \right\rangle.
\end{equation}
Supposing now that $\delta=3$, by the expression \eqref{formula_sigma_ab_unconst}, we see that the integrand in right-hand side of \eqref{exp_fst_part_ibpf_a_3_unconst} equals
\[
- K(a, m) \, h_r \, \phi_r^{-3} \, f(0),
\]
which coincides with the quantity \eqref{right_hand_side_ibpf_unconst} for $\delta=3$. Finally, supposing that $\delta=1$, by \eqref{formula_sigma_ab_unconst}, we see that the integrand in the right-hand side of \eqref{exp_fst_part_ibpf_a_1_unconst} equals
\[
 K(a, m) \, h_r \, \phi_r^{-3} \, f'(0),
\]
which also coincides with the quantity \eqref{right_hand_side_ibpf_unconst} with $\delta=1$. In conclusion, comparing the expressions \eqref{left_hand_side_ibpf_unconst} and \eqref{right_hand_side_ibpf_unconst}, we see that the claimed IbPF then follows as a consequence of the following result, the proof of which is postponed to the Appendix \ref{Proofs_ibpf}:

\begin{lm}\label{unconstrained_secondder}
For all $t > 0$ and $a \geq 0$, we have
\begin{equation*}
\frac{\d^{2}}{\d t^{2}} E^{\delta}_{a} \left( {X_{t}} \right) = -\Gamma \left( \frac{\delta+1}{2} \right) \left\langle \mu_{\frac{\delta-3}{2}}(y) \, , \, \frac{q^{\delta}_{t}(a^2,y)}{y^{\delta/2-1}} \right\rangle.
\end{equation*}
\end{lm}

\end{proof}

\subsection{The case of bridges}

Now we finally prove the IbPF associated with Bessel bridges stated in Theorem \ref{thm_ibpf_positive_cond}. This will follow from Theorem \ref{thm_ibpf_positive_cond_unconst} by simply conditioning on the value of $X_{1}$.

\begin{proof}[Proof of Theorem \ref{thm_ibpf_positive_cond}]
Let $\Phi \in S$ and $h \in C^{2}_{c}(0,1)$. Then, for any $\lambda \geq 0$, we consider the functional $\Psi: C([0,1])\to \mathbb{R}$ defined as
\[ \Psi(X) := \Phi(X) \, e^{-\lambda X_{1}^{2}}, \quad X \in C([0,1]). \]
Note that $\Psi$ is an element of $S$, since one can write $X_{1}^{2} := \int_0^1 X_t^2 \d m(X)$, where $m := \delta_{1}$ is the Dirac measure at $1$. Therefore, $\Psi$ satisfies the IbPFs stated in Theorem \ref{thm_ibpf_positive_cond_unconst}. Moreover, since $h_{1}=0$, we have
\[ \forall X \in C([0,1]), \quad \partial_{h} \Psi(X) = \partial_{h} \Phi(X) e^{- \lambda X_{1}^{2}}. \]
Therefore, assuming for example that $\delta \notin \{ 1, 3 \}$, we have
\begin{equation}
\label{equality_with_exp}
\begin{split}
& E^{\delta}_{a} (\partial_{h} \Phi (X) e^{- \lambda X_{1}^{2}}) + E^{\delta}_{a} (\langle h '' , X \rangle \, \Phi(X) e^{- \lambda X_{1}^{2}})= 
 \\ &-\kappa(\delta)\int_{0}^{1}  
 h_{r}  \int_0^\infty b^{\delta-4} \Big[ \mathcal{T}^{\,2k}_{b} \, \Sigma^{\delta,r}_{a}(\Phi (X) e^{- \lambda X_{1}^{2}} \,|\, \cdot\,) \Big]
  \d b \d r.
\end{split}
\end{equation}
By conditioning on the value of $X_{1}$, we can rewrite the left-hand side of this equality as 
\[ \int_{0}^{\infty} p^{\delta}_{1}(a,a') \,  e^{-\lambda a'^{2}} \left( E^{\delta}_{a,a'} (\partial_{h} \Phi (X)) + E^{\delta}_{a,a'} (\langle h '' , X \rangle \, \Phi(X)) \right) \d a'. \]
On the other hand, for all $r \in (0,1)$ and $b \geq 0$, we have, by the same type of conditioning
 \[ E^{\delta}_{a}(\Phi (X) e^{- \lambda X_{1}^{2}} \, | \, X_{r} = b ) = \int_{0}^{\infty} p^{\delta}_{1-r}(b,a') \, e^{-\lambda a'^{2}} \,  E^{\delta}_{a,a'} (\Phi (X) | X_{r}=b) \d a', \] 
whence we deduce that
\[ \Sigma^{\delta,r}_{a}(\Phi (X) e^{- \lambda X_{1}^{2}} \,|\, b \,) = \int_{0}^{\infty} p^{\delta}_{1}(a,a') e^{-\lambda a'^{2}} \, \Sigma^{\delta,r}_{a,a'} ( \Phi (X) \,|\, b \,) \d a'. \] 
Consequently, the relation \eqref{equality_with_exp} above can be rewritten 
\[ \begin{split}
& \int_{0}^{\infty} p^{\delta}_{1}(a,a')  \, e^{-\lambda a'^{2}} \left( E^{\delta}_{a,a'} (\partial_{h} \Phi (X)) + E^{\delta}_{a,a'} (\langle h '' , X \rangle \, \Phi(X)) \right) \d a' = \\
& - \kappa(\delta) \int_{0}^{\infty}  p^{\delta}_{1}(a,a') \, e^{-\lambda a'^{2}} \left( \int_{0}^{1} h_{r} \int_{0}^{\infty} b^{\delta-4} \Big[ \mathcal{T}^{\,2k}_{b}\, \Sigma^{\delta,r}_{a,a'} ( \Phi (X) \,|\, \cdot \,) \Big] \d b \d r \right) \d a'.
\end{split}\]
Note that this equality holds for any $\lambda \geq 0$. Hence the functions 
\[ x \mapsto \frac{p^\delta_1(a,\sqrt{x})}{\sqrt{x}} \left( E^{\delta}_{a,\sqrt{x}} (\partial_{h} \Phi (X)) + E^{\delta}_{a,\sqrt{x}} (\langle h '' , X \rangle \, \Phi(X)) \right) \]
and
\[x \mapsto - \kappa(\delta) \, \frac{p^\delta_1(a,\sqrt{x})}{\sqrt{x}} \left( \int_{0}^{1} h_{r} \int_{0}^{\infty} b^{\delta-4} \Big[ \mathcal{T}^{\,2k}_{b}\, \Sigma^{\delta,r}_{a,\sqrt{x}} ( \Phi (X) \,|\, \cdot \,) \Big] \d b \d r \right) \]
have the same Laplace transform. Since they are continuous on $(0,\infty)$, they must coincide. This yields the claimed IbPF for Bessel bridges of dimension $\delta \notin \{ 1,3 \}$. The cases $\delta \in \{ 1,3 \}$ are treated in the same way.
\end{proof}

\section{The dynamics via Dirichlet forms for $\delta=2$}

\label{sect_dynamics}

The IbPFs obtained above, which complete the results already obtained in \cite{EladAltman2019}, bear out the conjectures \eqref{1<spde<3}, \eqref{spde=1} and \eqref{0<spde<1} above for the structure of the gradient dynamics associated with the laws of Bessel bridges of dimension smaller than 3. However, as stressed in the introduction and in Section 6 of \cite{EladAltman2019}, we are still far from being able to solve such equations. However, in Section 5 of \cite{EladAltman2019}, a solution to a weak form of \eqref{spde=1}, the $1$-Bessel SPDE with homogeneous Dirichlet boundary conditions, was constructed using Dirichlet form techniques. 

In this section we go one step further by treating the case $\delta=2$. In words, we exploit our IbPFs to construct a weak version of the gradient dynamics associated with the law of a $2$-dimensional Bessel bridge from $0$ to $0$ over $[0,1]$, using the theory of Dirichlet forms. 
The reason for considering this particular Bessel bridge is that for integer values of $\delta$, and for zero boundary conditions, we can exploit a representation of the $\delta$-dimensional Bessel bridge in terms of a Brownian bridge, for which the corresponding gradient dynamics is well-known and corresponds to a linear stochastic heat equation. 
This representation allows us to construct a quasi-regular Dirichlet form associated with $P^2:=P^2_{0,0}$, a construction which does not follow from the IbPF \eqref{exp_fst_part_ibpf_a_b} due to the distributional character of its last term. The IbPF \eqref{exp_fst_part_ibpf_a_b} is then exploited to prove that the associated Markov process satisfies \eqref{formal2}, in a certain sense to be made precise below.
The proofs will follow closely those of Section 5 of \cite{EladAltman2019} which treated the case $\delta=1$. 
\subsection{The $2$-dimensional random string}

Consider the space $\mathbb{H}_{2} := L^{2}([0,1], \mathbb{R}^{2})$ endowed with the component-wise $L^{2}$ product. Let $\mu_{2}$ denote the law, on $\mathbb{H}_{2}$, of a two-dimensional Brownian bridge from $0$ to $0$. We shall use the shorthand notation $L^{2} (\mu_2)$ for the space $L^2(\mathbb{H}_{2},\mu_2)$. Consider moreover the semigroup $(\mathbf{Q}^{2}_{t})_{t \geq 0}$ on $\mathbb{H}_{2}$ defined, for all $F \in L^{2} (\mathbb{H}_{2},\mu_{2})$, and $z =(z_{1},z_{2}) \in \mathbb{H}_{2}$, by
\[ \mathbf{Q}^{2}_{t} F (z) := \mathbb{E} \left[ F(v_{t}(z)) \right], \quad t \geq 0, \]
where $(v_{t}(z))_{t \geq 0}$ is the solution to the $2$-dimensional stochastic heat equation with initial condition $z$ and with homogeneous Dirichlet boundary conditions
\[ \begin{split}
\begin{cases}
\frac{\partial v}{\partial t} = \frac{1}{2} \frac{\partial^{2} v}{\partial x^{2}} + \xi \\
v(0,x) = z(x), \qquad x\in[0,1], \\
v(t,0)= v(t,1)= 0,  
\end{cases}
\end{split}\]
where $\xi := (\xi_{1}, \xi_{2})$, with $\xi_{1}, \xi_{2}$ two independent space-time white noises on $\mathbb{R}_{+} \times [0,1]$. More precisely, let $(g_t(x,y))_{t \geq 0, x,x' \in (0,1)}$ be the fundamental solution of the heat equation on $[0,1]$ with homogeneous Dirichlet boundary conditions,  which by definition is the unique solution to
\[ \begin{split}
\begin{cases}
\frac{\partial g}{\partial t} = \frac{1}{2} \frac{\partial^{2} g}{\partial x^{2}} \\
g_{0}(x,x') = \delta_{x}(x') \\
g_{t}(x,0)= g_{t}(x,1)= 0.  
\end{cases}
\end{split}\]
Recall that $g$ can be represented as follows:
\begin{equation} 
\label{representation_green_function}
\forall t >0, \quad \forall x, x' \geq 0, \quad g_{t}(x,x') = \sum_{k=1}^{\infty} e^{-\frac{\lambda_{k}}{2}t} e_{k}(x) e_{k}(x'), 
\end{equation}
where $(e_{k})_{k \geq 1}$ is the complete orthornormal system of $H$ given by
\[e_{k}(x) := \sqrt{2} \sin(k \pi x), \quad x \in [0,1], \quad k \geq 1\] 
and $\lambda_{k} := k^{2} \pi^{2}$, $k \geq 1$. The process $v$ above can then be written as follows:
\[ 
v(t,x) = (v_{1}(t,x), v_{2}(t,x)), \quad t \geq 0, \, \, x \in [0,1],\]
where, for $i=1,2$
\[ 
v_{i}(t,x) = z_{i}(t,x) + \int_{0}^{t} \int_{0}^{1} g_{t-s}(x,x') \, \xi_{i} ({\rm d} s, \d x'), \]
with $z_{i}(t,x) := \int_{0}^{1} g_{t}(x,x') z_{i}(x') \d x'$, and where the integral above is a stochastic convolution. In words, $v$ is the vector composed of two independent copies of a solution to the one-dimensional stochastic heat equation, with respective intial data $z_{1}$ and $z_{2}$. In particular, it follows from this fomula that $v$ is a Gaussian process. An important role will be played by its covariance function. Namely, for all $t \geq 0$ and $x,x' \in (0,1)$, we set
\[ q_{t}(x,x') := \text{Cov}(v_1 (t,x) , v_1 (t,x')) = \text{Cov}(v_2 (t,x) , v_2 (t,x')) = \int_{0}^{t} g_{2 \tau}(x,x') \d \tau, \]
We also set
\[  q_{\infty}(x,x') := \int_{0}^{\infty} g_{2 \tau}(x,x') \d \tau=\mathbb{E}[\beta_x\beta_{x'}]
= x \wedge x' - x x' . \]
For all $t \geq 0$, we moreover define
\[ q^{t}(x,x') := q_{\infty}(x,x') - q_{t}(x,x') = \int_{t}^{\infty} g_{2 \tau}(x,x') \d \tau.\]
When $x=x'$, we will use the shorthand notations $q_{t}(x), q_{\infty}(x)$ and $q^{t}(x)$ instead of $q_{t}(x,x), q_{\infty}(x,x)$ and $q^{t}(x,x)$ respectively. We denote by $(\Lambda^2,D(\Lambda^2))$ the Dirichlet form generated by $(\mathbf{Q}^{2}_{t})_{t \geq 0}$ in $L^{2} (\mathbb{H}_{2},\mu_{2})$, which is given by
\[ \Lambda^2(F,G) = \frac{1}{2} \int_{\mathbb{H}_{2}} \langle \overline {\nabla} F, \overline {\nabla} G \rangle_{\mathbb{H}_{2}} \d \mu_{2}, \quad F,G \in D(\Lambda^2) = W^{1,2}(\mu_{2}),
\]
where $W^{1,2}(\mu_{2}) \subset L^{2} (\mathbb{H}_{2},\mu_{2})$ is the Sobolev space associated with $\mu_2$ and, for all $F \in W^{1,2}(\mu_{2})$, $\overline {\nabla} F : \mathbb{H}_{2} \to \mathbb{H}_{2} $ denotes the gradient of $F$ in $\mathbb{H}_{2}$,  see Section 9 in \cite{dpz3}.

\subsection{Gradient Dirichlet form associated with the 2-dimensional Bessel bridge}


As in Section 5 of \cite{EladAltman2019}, we set $H := L^{2}(0,1)$ and denote by $\langle \cdot, \cdot \rangle$ the $L^2$ inner product on $H$:
\[ \langle f, g \rangle := \int_0^1 f_r \, g_r \, \d r, \quad f,g \in H. \]
We denote by $\| \cdot \|$ the corresponding norm on $H$. We also consider the closed subset $K \subset H$ of nonnegative functions:
\[K:= \{ z \in H, \, \, z \geq 0 \, \, \text{a.e.} \}. \] 
Recall that $K$ is a Polish space. We further denote by $\nu_2$ the law, on $K$, of a $2$-dimensional Bessel bridge from $0$ to $0$ on $[0,1]$ (so that $P^2$ is then the restriction of $\nu_2$ to $C([0,1])$). We shall use the shorthand $L^2(\nu)$ to denote the space $L^2(K,\nu)$. Let $\mathcal{F} \mathcal{C}^{\infty}_{b}(H)$ denote the space of all functionals $F:H \to \mathbb{R}$ of the form
\[ F (z) = \psi(\langle l_{1}, z \rangle, \ldots, \langle l_{m}, z \rangle ), \quad z \in H, \] 
with $m \in \mathbb{N}$, $\psi \in C^{\infty}_{b}(\mathbb{R}^{m})$, and $l_{1}, \ldots, l_{m} \in \text{Span} \{ e_{k}, k \geq 1 \}$.  
We also define:
\[\mathcal{F} \mathcal{C}^{\infty}_{b}(K) := \{ F \big \rvert_{K} \, , \quad F \in \mathcal{F} \mathcal{C}^{\infty}_{b}(H) \}. 
\]
Moreover, for $f \in \mathcal{F} \mathcal{C}^{\infty}_{b}(K)$ of the form $f=F \big \rvert_{K}$, with $F \in \mathcal{F} \mathcal{C}^{\infty}_{b}(H)$, we define $\nabla f : K \to H$ by
\[ \nabla f (z) = \nabla F(z), \quad z \in K, \]
where this definition does not depend on the choice of $F\in \mathcal{F} \mathcal{C}^{\infty}_{b}(H)$ such that $f=F\big \rvert_{K}$.
We denote by $\mathcal{E}^2$ the bilinear form defined on $\mathcal{F} \mathcal{C}^{\infty}_{b}(K)$ by
\[ \forall f,g \in \mathcal{F} \mathcal{C}^{\infty}_{b}(K), \quad \mathcal{E}^2(f,g) := \frac{1}{2} \int_K \langle \nabla f , \nabla g \rangle \d \nu_{2}. 
\]
Finally, we denote by $j_{2} : \mathbb{H}_{2} \to K$ the map
\begin{equation}
\label{euclidean_norm_map}
j_{2}(z) := \|z\|=\sqrt{(z_{1})^{2} + (z_{2})^{2}}, \quad z=(z_{1},z_{2}) \in \mathbb{H}_{2}.
\end{equation}
Note that 
\begin{equation}
\label{image_measure}
\nu_2 = \mu_2 \circ j_{2}^{-1},
\end{equation}
so that the map
\[\begin{cases}
L^{2}(\nu_2) \to L^{2}(\mu_2) \\
\varphi \mapsto \varphi \circ j_2
\end{cases} \]
is an isometry. The following proposition can then be proven similarly as Theorem 5.1.3 in \cite{vosshallthesis} or Prop. 5.1 in \cite{EladAltman2019}: 
\begin{prop}
\label{closability_2}
The form $(\mathcal{E}^2,\mathcal{F} \mathcal{C}^{\infty}_{b}(K))$ is closable. Its closure $(\mathcal{E}^2,D (\mathcal{E}^2))$ is a local, quasi-regular Dirichlet form on $L^{2}(\nu_2)$. Moreover, for all $f \in D (\mathcal{E}^2)$, $f \circ j_2 \in D(\Lambda^2)$, and we have:
\begin{equation}
\label{isometry}
\forall f,g \in D(\mathcal{E}^2), \quad \mathcal{E}^2(f,g) = \Lambda^2(f \circ j_2, g \circ j_2) 
\end{equation}
\end{prop}
Let $(Q^2_t)_{t \geq 0}$ be the contraction semigroup on $(K,\nu_2)$ associated with the Dirichlet form $(\mathcal{E}^2, D(\mathcal{E}^2))$. Let also $\mathcal{B}_{b}(K)$ denote the set of Borel and bounded functions on $K$. As a consequence of Prop. \ref{closability_2}, in virtue of Thm IV.3.5 and Thm V.1.5 in \cite{ma2012introduction}, we obtain the following:

\begin{cor}
There exists a Markov diffusion process
\[M=\{\Omega, \mathcal{F}, (u_t)_{t \geq 0}, (\mathbb{P}_x)_{x \in K} \} \]
properly associated to $(\mathcal{E}^2,D(\mathcal{E}^2))$, i.e. for all $\varphi \in L^{2}(\nu_2) \cap \mathcal{B}_b(K)$, and for all $t > 0$,  $E_{x}(\varphi(u_{t})), \, x \in K$ defines an $\mathcal{E}^2$ quasi-continuous version of $Q^2_{t} \varphi$. Moreover, the process $M$ admits the following continuity property:
\[P_{x}[t \mapsto u_{t} \, \, \text{is continuous on} \, \, \mathbb{R}_{+}] = 1,\quad \text{for} \, \, \mathcal{E}^2 \, q.e. \, x  \in K. \]
\end{cor}

The rest of this section will be devoted to show that for $\mathcal{E}^2$ q.e. $x \in K$, under $\mathbb{P}_{x}$, $(u_t)_{t\geq 0}$ solves \eqref{1<spde<3} with $\delta=2$, or rather its weaker form \eqref{formal2}. An important technical point is the density of the space $\mathcal{S}$ introduced in Section \ref{sect_ibpf} above in the Dirichlet space $D(\mathcal{E}^2)$. To state this precisely, as in Section 5 of \cite{EladAltman2019}, we consider $\mathscr{S}$ the vector space generated by functionals $F:H \to \mathbb{R}$ of the form
\[ F(\zeta) = \exp(- \langle \theta,\zeta^{2} \rangle), \quad \zeta \in H,  \] 
for some $\theta : [0,1] \to \mathbb{R}_{+}$ Borel and bounded. Note that $\mathscr{S}$ may be seen as a subspace of the space  $\mathcal{S}$ of Section \ref{sect_ibpf} in the following sense: for any $F \in \mathscr{S}$, $F \rvert_{C([0,1])} \in \mathcal{S}$. We also set
\[ \mathscr{S}_{K} := \{ F \big \rvert_{K}, \quad F \in \mathscr{S} \}.  \]

\begin{lm}
\label{density2}
$\mathscr{S}_{K}$ is dense in $D (\mathcal{E}^2)$.
\end{lm}

\begin{proof}
The same arguments as in the proof of Lemma 5.3 in \cite{EladAltman2019}, showing that $\mathscr{S}_{K}$ is dense in $D (\mathcal{E})$, apply here. Indeed, the only particular feature of the space $D (\mathcal{E})$ used in the proof of Lemma 5.3 is the fact that $\nu$ has finite second moments, that is $\int_{K} \| x \|^{2} \d \nu (x) < \infty$.
Since the same is true for $\nu_{2}$ in place of $\nu$, the same arguments apply for $D (\mathcal{E}^2)$ in place of $D (\mathcal{E})$, and the claim follows.
\end{proof}

\subsection{Convergence of one-potentials}

In order to show that the Markov process $(u_t)_{t \geq 0}$ constructed above satisfies an equation of the form \eqref{formal2}, we will exploit the IbPF \eqref{exp_fst_part_ibpf_a_b} for $\delta=2$ (and $a=a'=0$). Here, as for $\delta=1$, we again face the caveat that the last term in that IbPF is distributional, so providing it with a dynamic interpretation requires some care. We will follow the same route as in Section 5 of \cite{EladAltman2019}, by approximating that distributional term with a family of smooth measures, and by showing that the convergence also holds for the associated one-potentials (see Section 5 of \cite{fukushima2010dirichlet} for the definition of one-potentials). This will enable us to identify the drift term in the SPDE as a limit in probability of smooth drifts. Recall that we have the following equality in law on $K$:
\begin{equation} 
\label{relation_law_bridge_bessel}
(\|\beta_{t}\|)_{0 \leq t \leq 1} \overset{(d)}{=} \nu_2, 
\end{equation}
where $\beta=(\beta^{1},\beta^{2})$ is a two-dimensional Brownian bridge from $0$ to $0$, and $\|\cdot\|$ is the Euclidean norm on $\mathbb{R}^{2}$. Let now $\rho$ be a smooth function supported on $[-1,1]$ such that
\begin{equation}
\label{def_mollifier}
\rho \geq 0, \quad \int_{-1}^{1} \rho = 1, \quad \rho(y) = \rho(-y), \quad y \in \mathbb{R}, 
\end{equation}
and let us set, for all $\eta>0$,
\[\rho_{\eta}(y) := \frac{1}{\eta} \, \rho \left( \frac{y}{\eta} \right), \quad y \in \mathbb{R}. 
\]
Then, for any $\Phi \in \mathcal{S}$, the last term of the right-hand side in the IbPF  \eqref{exp_fst_part_ibpf_a_b} with $\delta=2$ and $a=a'=0$ can be re-expressed using the equality
\begin{equation}
\label{relation_cond_lt_2}
\begin{split}
&- \kappa(2) \int_{0}^{1} h_{r} \int_{0}^{\infty} \d b \, b^{-2}\left( \Sigma^{2,r}_{0,0}\left(\Phi(X) \,|\, b\right) - \Sigma^{2,r}_{0,0}\left(\Phi(X) \,|\, 0\right) \right)\d r  \\
&= \frac{1}{4} \, \lim_{\epsilon \to 0} \lim_{\eta \to 0} \mathbb{E} \left[ \Phi(\|\beta\|)  \int_{0}^{1} h_{r} \, \left( \frac{\mathbf{1}_{\{\|\beta_{r}\| \geq \epsilon\}}}{\|\beta_{r}\|^{3}} - \frac{2}{\epsilon} \frac{\rho_{\eta}(\|\beta_{r}\|)}{\|\beta_{r}\|} \right) \d r \right]. 
\end{split}
\end{equation}
Indeed, the equality follows upon noting that the process $X := \|\beta\|$ is distributed as $\nu_2$, and by conditioning on the value of $X_r$, $r \in (0,1)$. Let now $\epsilon, \eta >0$ with $\eta < \epsilon$ be fixed. We consider the functional $G_{\epsilon, \eta} : \mathbb{H}_{2} \to \mathbb{R}$ defined, for all $z=(z^{1},z^{2}) \in \mathbb{H}_{2}$, by
\[ G_{\epsilon, \eta}(z) := \int_{0}^{1} \d r \ h_{r} \, f_{\epsilon, \eta}(\|z_{r}\|), \]
where
\[f_{\epsilon, \eta}(x) := \frac{1}{4} \left( \frac{\mathbf{1}_{\{x \geq \epsilon\}}}{x^{3}} - \frac{2}{\epsilon} \frac{\rho_{\eta}(x)}{x}\right), \quad x \geq 0.\]
For all $\epsilon > \eta >0$, we then define the functional $V_{\epsilon, \eta }: \mathbb{H}_{2} \to \mathbb{R}$ by
\[ V_{\epsilon, \eta} (z)= \int_{0}^{\infty} e^{-t} \,  \mathbf{Q}^{2}_{t} G_{\epsilon, \eta}(z) \d t, \quad z \in H.\]
Note that, in the language of  Chap. 5 of \cite{fukushima2010dirichlet}, $V_{\epsilon, \eta}$ is the one-potential associated with the continuous additive functional
\[ \int_{0}^{t} G_{\epsilon, \eta}(v(s,\cdot)) \d s, \quad t \geq 0. \]
In particular, $V_{\epsilon, \eta} \in D(\Lambda^{2})$. We will show that $V_{\epsilon,\eta}$ converges in $D(\Lambda^2)$ as we send first $\eta$, and then $\epsilon$, to $0$. To do so, we remark that, for all $z \in \mathbb{H}_{2}$, we have
\[ V_{\epsilon, \eta}(z) = \int_{0}^{\infty} \d t \int_{0}^{1} \d r \frac{e^{-t} \,h_{r}}{8 \pi q_{t}(r)} \int_{\mathbb{R}^{2}} \d a \left( \frac{\mathbf{1}_{ \|a\| \geq  \epsilon}}{\|a\|^{3}} - \frac{2}{\epsilon} \frac{\rho_{\eta}(\|a\|)}{\|a\|} \right) \exp \left( - \frac{\|a-z(t,r)\|^{2}}{2 q_{t}(r)} \right),\]
where, for all $t >0$ and $r \in (0,1)$, $z(t,r):= (z_{1}(t,r), z_{2}(t,r)) =  \int_{0}^{1} g_{t}(r,r') \, z(r') \d r'$. We define also the functional $V_{\epsilon} : \mathbb{H}_{2} \to \mathbb{R}$ by setting, for all $z \in \mathbb{H}_{2}$,
\[
\begin{split}
V_{\epsilon}(z) &:= \int_{0}^{\infty} {\rm d}t \, e^{-t} \int_{0}^{1} \frac{h_{r}}{8 \pi q_{t}(r)} \d r \int_{\mathbb{R}^{2}} \d a \frac{\mathbf{1}_{\|a\| \geq \epsilon}}{\|a\|^{3}} \cdot \\
& \cdot \left( \exp \left( - \frac{\|a-z(t,r)\|^{2}}{2 q_{t}(r)} \right) - \exp \left( - \frac{\|z(t,r)\|^{2}}{2 q_{t}(r)} \right) \right) . 
\end{split}
\]
Note that, by splitting the domain $\mathbb{R}^{2}$ into four quadrants, we can rewrite
\begin{equation}
\label{intermediate_one_pt}
\begin{split}
V_{\epsilon}(z) &:= \int_{0}^{\infty} {\rm d}t \, e^{-t} \int_{0}^{1} \frac{h_{r}}{8 \pi q_{t}(r)}  \d r\int_{\mathbb{R}_{+}^{2}} \d a \, \frac{\mathbf{1}_{\|a\| \geq \epsilon}}{\|a\|^{3}} \cdot \\
 & \cdot \sum_{\alpha \in \{ -1, 1 \}^{2}} \left( \exp \left( - \frac{\|\alpha a-z(t,r)\|^{2}}{2 q_{t}(r)} \right) - \exp \left( - \frac{\|z(t,r)\|^{2}}{2 q_{t}(r)} \right) \right), 
\end{split}
\end{equation}
where, for all $a \in \mathbb{R}^{2}$ and $\alpha \in \{ -1, 1 \}^{2} $, we have set
\[ \alpha \, a := (\alpha_{1} a_{1}, \alpha_{2} a_{2}). \]
Let us finally define the functional $V: \mathbb{H}_{2} \to \mathbb{R}$ by setting, for all $z \in \mathbb{H}_{2}$
\begin{equation}
\label{limiting_one_pt_bis}
\begin{split}
 V(z) &:= \int_{0}^{\infty} {\rm d}t \, e^{-t} \int_{0}^{1} \frac{h_{r}}{8 \pi q_{t}(r)} \d r \int_{\mathbb{R}_{+}^{2}} \frac{\d a}{\|a\|^{3}} \cdot \\
 & \cdot \sum_{\alpha \in \{ -1, 1 \}^{2}} \left( \exp \left( - \frac{\|\alpha a-z(t,r)\|^{2}}{2 q_{t}(r)} \right) - \exp \left( - \frac{\|z(t,r)\|^{2}}{2 q_{t}(r)} \right) \right) .
\end{split} 
\end{equation}
We then have the following result, the proof of which is postponed to the Appendix \ref{Proofs_dynamics}.

\begin{prop}
\label{conv_one_pot}
The functionals $V_{\epsilon}$ and $V$ all belong to $D(\Lambda^2)$. Moreover:
\[ \forall \epsilon >0, \quad  \underset{\eta \to 0}{\lim} \, V_{\epsilon,\eta} = V_{\epsilon},\]
and
\[\underset{\epsilon \to 0}{\lim} \, V_{\epsilon} = V,\] 
where all convergences take place in $D(\Lambda^2)$.
\end{prop}

\subsection{From the IbPFs to the dynamics} 

Using the above results as well as the IbPF for $P^{2}$, and arguing as in Section 5 of \cite{EladAltman2019}, we can now obtain an identification result for the dynamics of the Markov process $M$ constructed above. 
In the sequel, we set $\Lambda^{2}_{1} := \Lambda^{2} + (\cdot,\cdot)_{L^{2}(\mu_2)}$ and $\mathcal{E}^{2}_{1} := \mathcal{E}^{2} + (\cdot,\cdot)_{L^{2}(\nu_2)}$. As for the case $\delta=1$ considered in Section 5 of \cite{EladAltman2019}, we shall exploit a projection principle, the proof of which follows exactly as for Lemma 5.5 in \cite{EladAltman2019}:

\begin{lm}
\label{projection2}
There exists a unique bounded linear operator $\Pi: D(\Lambda^{2}) \to D(\mathcal{E}^{2})$ such that, for all $F,G \in D(\Lambda^{2})$ and $f \in D(\mathcal{E}^{2})$,
\[ \Lambda^{2}_{1}(F,f \circ j_{2}) = \mathcal{E}^{2}_{1}(\Pi F,f), \]
where $j_{2}$ is as in \eqref{euclidean_norm_map}. Moreover, we have:
\[\mathcal{E}^{2}_{1} (\Pi F,\Pi F) \leq \Lambda^{2}_{1}(F,F). \]
\end{lm}
As a consequence, we can obtain a stronger version of the IbPF for $P^{2}$. Recall that, by Prop \ref{conv_one_pot} and by the above definitions, $V = \underset{\epsilon \to 0}{\lim} \, \underset{\eta \to 0}{\lim} \, V_{\epsilon, \eta}$, where $V_{\epsilon, \eta} \in D(\Lambda^2)$ is the one-potential associated with the additive functional
\[  \int_0^t \d s \, G_{\epsilon, \eta}(v(s, \cdot)) = \int_0^t \d s \int_0^1 \d r \, h_r f_{\epsilon, \eta}(\|v(s,r)\|) \d r \d s. \] 
Therefore, combining the IbPF \eqref{exp_fst_part_ibpf_a_b} with $\delta=2$ and $a=a'=0$, the equality \eqref{relation_cond_lt_2}, the density result \ref{density2} and the projection principle \ref{projection2}, and arguing as for the proof of Corollary 5.6 in \cite{EladAltman2019}, we obtain the following result:

\begin{cor}
For all $f \in D(\Lambda^{2})$ and $h \in C^{2}_{c}(0,1)$, we have
\begin{equation}
\label{IbPF_Dirichlet_2}
\mathcal{E}^{2}\left(\langle h, \cdot \rangle - \frac{1}{2} \Pi V \, ,\, f\right) = - \frac{1}{2} \int_K (\langle h'', \zeta \rangle - \Pi  V(\zeta)) f(\zeta) \d \nu_2 (\zeta), 
\end{equation}
where $V$ is as in \eqref{limiting_one_pt_bis}.
\end{cor}

 
Recall that $M=\{\Omega, \mathcal{F}, (u_t)_{t \geq 0}, (\mathbb{P}_x)_{x \in K} \}$ denotes the Markov process properly associated with the Dirichlet form $(\mathcal{E}^2,D (\mathcal{E}^2))$ constructed above. 
The following theorem says that the process $(u_t)_{t \geq 0}$ satisfies the SPDE \eqref{formal2} above, which is a weaker version of the Bessel SPDE \eqref{1<spde<3} with $\delta=2$.

\begin{thm}
\label{weak_spde_2}
For all $h \in C^{2}_{c}(0,1)$, we have, almost surely
\[ \langle u_{t}, h \rangle - \langle u_0, h \rangle = M_{t} + N_{t}, \quad t \geq 0, \qquad \mathbb{P}_{u_0}-\text{a.s.}, \quad \text{q.e.} \ u_0 \in K.  \]
Here $(N_{t})_{t \geq 0}$ is a continuous additive functional of zero energy satisfying
\[N_{t} = \frac{1}{2} \int_{0}^{t} \langle h'', u_{s} \rangle \d s - \lim_{\epsilon \to 0} \lim_{\eta \to 0}  \, N^{\epsilon, \eta}_{t}, \]
where
\[N^{\epsilon, \eta}_{t} = \frac{1}{2} \int_{0}^{t} \langle f_{\epsilon,\eta}(u_{s}), h \rangle \d s, \]
and where the limit holds in $\mathbb{P}_{\nu_2}$ probability, for the topology  of uniform convergence in $t$ on each finite interval. Moreover, $(M_{t})_{t \geq 0}$ is a martingale additive functional whose sharp bracket has Revuz measure $\|h\|_{L^{2}}^{2} \, \nu_2$. 
\end{thm}

\begin{proof}
The result follows from \eqref{IbPF_Dirichlet_2} using the same arguments as in the proof of Theorem 5.9 in \cite{EladAltman2019}.
\end{proof}

\subsection{A transition in the dynamics at $\delta=2$?}

The $2$-Bessel SPDE studied above is believed to display a very interesting behavior at $0$. Thus, we believe that $2$ should be the critical value of $\delta>0$ for the probability of a solution to the $\delta$-Bessel SPDE to hit $0$: see section 6 of \cite{EladAltman2019}. More precisely, extrapolating the main result of \cite{dalang2006hitting}, we conjecture that, for all $\delta > 2$, a solution to the $\delta$-Bessel SPDE a.s. hits the obstacle $0$ in at most $\lfloor\frac 4{\delta-2}\rfloor$ space points simultaneously in time, while for $\delta < 2$ it may vanish at infinitely many space points simultaneously with positive probability. This would be in agreement with the fact that $\delta=2$ is also the critical dimension for the probability that the $\delta$-Bessel process or bridge hit $0$. In the particular case $\delta=2$, we conjecture that the solution $(u_{t})_{t \geq 0}$ can hit an arbitrary number of space points simulatneously, that is, for all $k \in \mathbb{N}$,
\[\mathbb{P}(\{\exists \, t > 0, \exists \, x_1 < \ldots < x_k \, | \, u_t(x_1, \ldots, x_k) =0\}) > 0 . \]
 Note that this was proven in \cite{dalang2006hitting} (see Theorem 2.4 therein) in the case of a stationnary 2-dimensional pinned string. More precisely, defining  $U_t(x) \in \mathbb{R}^2$, $t \geq 0, x \in \mathbb{R}$ to be the stationary solution of the 2-dimensional vector-valued heat equation on $\mathbb{R}_+ \times \mathbb{R}$, then for any $k \in \mathbb{N}$, 
\[\mathbb{P}(\{\exists \, t > 0, \exists \, x_1 < \ldots < x_k \, | \, U_t(x_1, \ldots, x_k) =0\}) > 0 . \]
This tends to support our conjecture on the behavior at $0$ of the Bessel SPDE of parameter $\delta=2$. Note however that the process $(\|U_t\|)_{t \geq 0}$, where $\| \cdot \|$ denotes the Euclidean norm in $\mathbb{R}^2$, does \textit{not} coinicide with the Markov process constructed in Section \ref{sect_dynamics} above, as one would like to infer by analogy with the finite-dimensional setting. Although we did not provide a proof of this fact, it could be shown similarly as Theorem 5.9 in \cite{EladAltman2019}, which treated the case $\delta=1$. Therefore the precise hitting properties of the Bessel SPDE of parameter $\delta=2$ seems to be a very open, subtle and intriguing question.

\appendix

\section{Proof of a technical lemma from Section \ref{sect_ibpf}}\label{Proofs_ibpf}

\begin{proof}[Proof of Lemma \ref{unconstrained_secondder}]

We have
\begin{align*} 
E^{\delta}_{a} \left( {X_{t}} \right) =Q^{\delta}_{x} \left( \sqrt{X_{t}} \right) &= \int_{0}^{\infty} y^{\frac{\delta+1}{2} -1} \left(y^{1-\delta/2} q^{\delta}_{t}(x,y) \right)  \d y \\
&=  \Gamma \left( \frac{\delta+1}{2} \right) \left\langle \mu_{\frac{\delta+1}{2}} (y), \, y^{1-\delta/2} q^{\delta}_{t}(x,y)  \right\rangle.
\end{align*}
Therefore, differentiating in $t$, we obtain
\[ 
\frac{\d}{\d t} \, Q^{\delta}_{x} \left( \sqrt{X_{t}} \right) =  \Gamma \left( \frac{\delta+1}{2} \right) \left\langle \mu_{\frac{\delta+1}{2}} (y), \partial_{t} \left(y^{1-\delta/2} q^{\delta}_{t}(x,y) \right) \right\rangle.
\]
Note that exchanging $\frac{\d}{\d t}$ and the brackets is justified by the fact that the function
\[ y \mapsto \partial_{t} \left(y^{1-\delta/2} q^{\delta}_{t}(x,y) \right)  \] 
is in $\mathcal{S}([0,\infty))$, so that
\[ \int_{0}^{\infty} y^{\frac{\delta+1}{2}-1}\left|\partial_{t} \left(y^{1-\delta/2} q^{\delta}_{t}(x,y) \right)\right| \d y <\infty.  \]
Now, we intend to re-express the time-derivative of $q^{\delta}_{t}(x,y)$. To do so, we recall that, by Kolmogorov's equation for the SDE satisfied by squared Bessel processes, the following equation holds
\begin{equation}
\label{pde_densities_only_adjoint} 
\partial_{t} q^{\delta}_{t}(x,y) = (4-\delta) \, \partial_{y} q^{\delta}_{t}(x,y) + 2 y \, \partial^{2}_{y} q^{\delta}_{t}(x,y). 
\end{equation}
By the Leibniz formula, this equality can be rewritten  
\begin{equation}
\label{deriv_qt_unconstrained} 
\begin{split}
\partial_{t} \left(\frac{q^{\delta}_{t}(x,y)}{y^{\delta/2-1}} \right) = 2 \, \partial_{y} \left(  y^{2-\delta/2} \, \partial_{y} q^{\delta}_{t}(x,y)\right), \qquad t \geq 0, \quad x,y \geq 0. 
\end{split} 
\end{equation} 
Hence, applying the distribution $\mu_{\frac{\delta+1}{2}}$ in the variable $y$, and recalling Proposition \ref{thm_ibpf_mu}, we obtain 
\begin{equation}
\label{fst_drivative}
\begin{split}
\frac{\d}{\d t} \, Q^{\delta}_{x} \left( \sqrt{X_{t}} \right) =  -2 \, \Gamma \left( \frac{\delta+1}{2} \right) \left\langle \mu_{\frac{\delta-1}{2}},  y^{2-\delta/2} \, \partial_{y} q^{\delta}_{t}(x,y) \right\rangle. 
\end{split} 
\end{equation} 
We now intend to re-express the right-hand side of \eqref{fst_drivative} without the derivative $\partial_{y}$. To do so, we note that 
\begin{equation}
\label{leibniz}
 y^{2-\delta/2} \, \partial_{y} q^{\delta}_{t}(x,y) =  \partial_{y} \left( y^{2-\delta/2} \,  q^{\delta}_{t}(x,y) \right) - \left(2 - \frac{\delta}{2} \right) y^{1-\delta/2} \, q^{\delta}_{t}(x,y). 
\end{equation}
Hence, by Proposition \ref{thm_ibpf_mu}, we have
\begin{equation}
\label{exp_fst_trm}
\begin{split}
& \left\langle \mu_{\frac{\delta-1}{2}},  y^{2-\delta/2} \, \partial_{y} q^{\delta}_{t}(x,y) \right\rangle \\
&= - \left\langle \mu_{\frac{\delta-3}{2}},  y^{2-\delta/2} \, q^{\delta}_{t}(x,y) \right\rangle - \left(2 - \frac{\delta}{2} \right) \left\langle \mu_{\frac{\delta-1}{2}}, y^{1-\delta/2} \, q^{\delta}_{t}(x,y) \right\rangle.
\end{split}
\end{equation}
Now, applying Lemma \ref{lemma_mu} with $\alpha = \frac{\delta-3}{2}$ and $f$ the smooth function defined by
\[f(y) := y^{2-\delta/2} \, q^{\delta}_{t}(x,y), \quad y \in \mathbb{R}_{+},\] 
we can rewrite equation \eqref{exp_fst_trm} as
\begin{equation}
\label{re_expressed_fst_trm} 
\begin{split}
\left\langle \mu_{\frac{\delta-1}{2}},  y^{2-\delta/2} \, \partial_{y} q^{\delta}_{t}(x,y) \right\rangle &= - \left( \frac{\delta-3}{2} + 2 - \frac{\delta}{2} \right) \left\langle \mu_{\frac{\delta-1}{2}}, y^{1-\delta/2} \, q^{\delta}_{t}(x,y) \right\rangle \\
&= - \frac{1}{2} \left\langle \mu_{\frac{\delta-1}{2}}, y^{1-\delta/2} \, q^{\delta}_{t}(x,y) \right\rangle.
\end{split}
\end{equation}
We thus obtain
\begin{equation*}
\begin{split}
\frac{\d}{\d t} \, Q^{\delta}_{x} \left( \sqrt{X_{t}} \right) &= \Gamma \left( \frac{\delta+1}{2} \right)  \left\langle \mu_{\frac{\delta-1}{2}}, \, y^{1-\delta/2} \, q^{\delta}_{t}(x,y) \right\rangle
 \end{split} 
\end{equation*}
Hence, differentiating  in $t$ a second time, we obtain 
\begin{equation}
\label{snd_deriv}
\begin{split}
\frac{\d^{2}}{\d t^{2}} \, Q^{\delta}_{x,y} \left( \sqrt{X_{t}} \right) = \Gamma \left( \frac{\delta+1}{2} \right) \left\langle \mu_{\frac{\delta-1}{2}}, \partial_{t} \left( y^{1-\delta/2} \, q^{\delta}_{t}(x,y) \right) \right\rangle
 \end{split} 
\end{equation}
The fact that we can differentiate in $t$ inside the brackets is justified as before. Now, recalling the differential relation \eqref{deriv_qt_unconstrained}, and applying Proposition \ref{thm_ibpf_mu}, we obtain
\begin{equation*}
\label{exp_deriv_fst_part_unconst}
\left\langle \mu_{\frac{\delta-1}{2}}, \partial_{t} \left( y^{1-\delta/2} \, q^{\delta}_{t}(x,y) \right) \right\rangle = - 2  \left\langle \mu_{\frac{\delta-3}{2}},  y^{2-\delta/2} \, \partial_{y} q^{\delta}_{t}(x,y) \right\rangle 
\end{equation*}
As previously, we re-express the right-hand side without derivatives. By \eqref{leibniz} and by Proposition \ref{thm_ibpf_mu}, we see that
\begin{equation*}
\left\langle \mu_{\frac{\delta-3}{2}},  y^{2-\delta/2} \, \partial_{y} q^{\delta}_{t}(x,y) \right\rangle  = - \left\langle \mu_{\frac{\delta-5}{2}},  y^{2-\delta/2} \, q^{\delta}_{t}(x,y) \right\rangle - \left(2 - \frac{\delta}{2} \right) \left\langle \mu_{\frac{\delta-3}{2}}, y^{1-\delta/2} \,  q^{\delta}_{t}(x,y) \right\rangle.
\end{equation*}
Upon applying Lemma \ref{lemma_mu} to the first term in the right-hand side, we thus obtain
\[\begin{split}
\left\langle \mu_{\frac{\delta-3}{2}},  y^{2-\delta/2} \, \partial_{y} q^{\delta}_{t}(x,y) \right\rangle &= - \left( \frac{\delta-5}{2} + 2 - \frac{\delta}{2}\right) \left\langle \mu_{\frac{\delta-3}{2}}, y^{1-\delta/2} \, q^{\delta}_{t}(x,y) \right\rangle \\
&= \frac{1}{2} \left\langle \mu_{\frac{\delta-3}{2}}, y^{1-\delta/2} \, q^{\delta}_{t}(x,y) \right\rangle
 \end{split}\]
Finally, we thus obtain
\[\frac{\d^{2}}{\d t^{2}} \, Q^{\delta}_{x} \left( \sqrt{X_{t}} \right) = - \Gamma \left( \frac{\delta+1}{2} \right) \left\langle \mu_{\frac{\delta-3}{2}}, y^{1-\delta/2} \, q^{\delta}_{t}(x,y) \right\rangle,\]
which yields the claim.
\end{proof}

\section{Proof of a technical result from Section \ref{sect_dynamics}}\label{Proofs_dynamics}

\begin{proof}[Proof of Proposition \ref{conv_one_pot}]
We first show that the sequence of functionals $(V_{\epsilon, \eta})_{0< \eta < \epsilon < 1}$ is bounded in $L^2(\mu_2)$. The proof of the requested convergences will follow by similar arguments. For any $t>0$, we have
\begin{equation}
\label{norm_sgp}
\| \mathbf{Q}^2_t G_{\epsilon, \eta} \|^2_{L^2(\mu)} = \int_{[0,1]^2} h_r \, h_s \, \int_{\mathbb{R}^2} f_{\epsilon,\eta}(\|a\|) f_{\epsilon,\eta}(\|b\|) \Gamma_{r,s}(a_1,b_1) \, \Gamma_{r,s}(a_2,b_2) \d a \d b, 
\end{equation}
where, for all $(r,s) \in [0,1]$ and $(x,y) \in \mathbb{R}^{2}$
\begin{equation}
\label{def_gamma_rs} 
\Gamma_{r,s}(x,y) := \mathbb{E} \left[ \frac{1}{2 \pi \sqrt{q_{t}(r) q_{t}(s)}} \exp \left(- \frac{(x-Z(t,r))^{2}}{2q_{t}(r)} - \frac{(y-Z(t,s))^{2}}{2q_{t}(s)} \right) \right]. 
\end{equation}
Here $Z(t,r)$, $t \geq 0, r \in [0,1]$ denotes the solution to the heat equation started from $W$,
\[Z(t,r) := \int_{0}^{1} g_{t}(r,r') W(r') \d r', \qquad t \geq 0, r \in [0,1], \]
with $(W(r))_{r \in [0,1]}$ a Brownian bridge on $[0,1]$, and we are taking the expectation with respect to $W$. Recall from the proof of Prop. 5.4 of \cite{EladAltman2019} that $\Gamma_{r,s}$ is the density of the centered Gaussian law on $\mathbb{R}^{2}$ with covariance matrix
\[ M = \begin{pmatrix}
q_{\infty}(r) & q^{t}(r,s) \\
q^{t}(r,s) & q_{\infty}(s)
\end{pmatrix}. \]
Here and below, we fix a constant $\vartheta > 0$ such that the test function $h$ is supported in $[\vartheta,1-\vartheta]$. Recall from (5.13) and (5.14) in \cite{EladAltman2019} that 
\begin{equation}
\label{lower_bound_det_trivial}
\det(M) \geq \vartheta^{2} \, |r-s| 
\end{equation}
and
\begin{equation}
\label{lower_bound_det} 
\det(M) \geq c_{\vartheta} \, (t \wedge 1) ^{1/2},
\end{equation} 
for some $c_{\vartheta}>0$ depending only on $\vartheta$. Now, in view of \eqref{norm_sgp}, it suffices to bound, for all $(r,s) \in (0,1)^{2}$, the integral
\[ I(\epsilon,\eta) :=  \int_{\mathbb{R}^{2} \times \mathbb{R}^{2}} f_{\epsilon, \eta}(||a||) \, f_{\epsilon, \eta}(||b||) \, \Gamma_{r,s}(a_{1},b_{1}) \, \Gamma_{r,s}(a_{2},b_{2}) \d a \d b.\]
Since $h$ is supported in $[\vartheta, 1-\vartheta]$, we may assume that $(r,s) \in [\vartheta, 1-\vartheta]^{2}$. To bound $I(\epsilon,\eta)$, we first switch to polar coordinates by setting
\[ a = (x \cos(\theta), x \sin(\theta)), \quad b = (y \cos(\varphi), y \sin(\varphi)), \]
with $x,y \geq 0$ and $\theta, \varphi \in [0,2\pi]$. We then have
\[ I(\epsilon,\eta) :=  \int_{\mathbb{R}_{+} \times \mathbb{R}_{+}} f_{\epsilon, \eta}(x) \, f_{\epsilon, \eta}(y) \, x y \, G(x,y) \d x \d y,\]
where 
\begin{equation}
\label{def_g}
 G(x,y) := \int_{[0,2 \pi]^{2}} \Gamma_{r,s}(x \cos(\theta), y \cos(\varphi)) \, \Gamma_{r,s}(x \sin(\theta), y \sin(\varphi)) \, \d \theta \d \varphi.
\end{equation}
Hence
\[\begin{split}
I(\epsilon,\eta) &= 16 \int_{\mathbb{R}_{+} \times \mathbb{R}_{+}} \left( \frac{\mathbf{1}_{\{x \geq \epsilon\}}}{x^{2}} - \frac{2}{\epsilon} \rho_{\eta}(x) \right) \left( \frac{\mathbf{1}_{\{y \geq \epsilon\}}}{y^{2}} - \frac{2}{\epsilon} \rho_{\eta}(y) \right) \, G(x,y) \d x \d y \\
&= 16 \int_{[\epsilon, \infty) \times [\epsilon, \infty)} \frac{\d x}{x^{2}}  \frac{\d y}{y^{2}} \, \Big( G(x,y) - \int_{[0,\eta]} 2 \rho_{\eta}(w) G(x,w) \d w - \\
& - \int_{[0,\eta]} 2 \rho_{\eta}(z) G(z,y) \d z + \int_{[0,\eta]^{2}} 4 \rho_{\eta}(z) \rho_{\eta}(w) G(z, w) \d z \d w \Big) \\
&=  (16)^2 \int_{[\epsilon, \infty) \times [\epsilon, \infty)} \frac{\d x}{x^{2}}  \frac{\d y}{y^{2}} \, \int_{[0,\eta]^{2}} \d z \d w \, \rho_{\eta}(z) \, \rho_{\eta}(w) \, G_{z,w} (x,y), 
\end{split}\]
with
\begin{equation}
\label{renormalized_g}
G_{z,w} (x,y) := G(x,y) - G(x,w) - G(z,y) + G(z,w). 
\end{equation}
Here we used the fact that $\int_0^\eta \rho_\eta(z) \d z = 1/2$ to obtain the last line. Hence, it suffices to bound, for all $x,y \geq \epsilon$ and $z,w \in [0,\eta]$, the quantity $G_{z,w} (x,y)$. By the triangular inequality, this is obviously bounded by:
\[ 4 \, \underset{x,y \geq 0}{\sup} |G(x,y)|. \]
In turn, by \eqref{def_g}, we have
\[ \begin{split}
 \underset{x,y \geq 0}{\sup} |G(x,y)| &\leq 4 \pi^{2} \underset{z \in \mathbb{R}^{2}}{\sup} |\Gamma_{r,s}(z)|^{2} \\ 
&\leq 4 \pi^{2} \frac{1}{4 \pi^2 \, \det(M)} \\
&\leq C_\vartheta \, (t \wedge 1)^{-1/2}
\end{split} \]
where $M$ denotes the covariant matrix associated with $\Gamma_{r,s}$, and where the last bound follows from \eqref{lower_bound_det}. Therefore, for all $x,y \geq \epsilon$ and $z,w \in [0,\eta]$
\begin{equation}
\label{bound_bigx_bigy}
|G_{z,w} (x,y)| \leq C_{\vartheta} \, (t\wedge1)^{-1/2}.
\end{equation} 
Here and below we are denoting by $C_\vartheta$ a positive constant which depends only on $\vartheta$ and which may change from line to line. Note that although the above bound is sufficient when $x$ and $y$ are away from $0$, say $x,y \geq 1$, it will not be satisfactory when either $x$ or $y$ tend to $0$: in that regime, we need a stronger bound in order to cure the potential divergency created by the terms $\frac{1}{x^{2}}$ and $\frac{1}{y^{2}}$ in the integral $I(\epsilon, \eta)$. This will be done by harvesting the renormalisations appearing in \eqref{renormalized_g}. Note that this kind of reasoning is a reminiscence (in a tremendously simpler context) of the sophisticated methods used to obtain bounds on Feynman integrals, for instance in the theory of regularity structure (see \cite[Appendix A]{hairer2015class}, and \cite{chandra2016analytic}). First note that, for all $x,y \geq 0$, we have
\[ \begin{split}
\frac{\partial G}{\partial x}(x,y) &= \int_{[0,2 \pi]^{2}} \Big( \cos(\theta) \frac{\partial \Gamma_{r,s}}{\partial x}(x \cos(\theta), y \cos(\varphi)) \, \Gamma_{r,s}(x \sin(\theta), y \sin(\varphi)) \\
&+ \sin(\theta) \Gamma_{r,s}(x \cos(\theta), y \cos(\varphi)) \, \frac{\partial \Gamma_{r,s}}{\partial x}(x \sin(\theta), y \sin(\varphi)) \Big) \d \theta \d \varphi, 
\end{split}\]  
whence we in particular obtain 
\[\frac{\partial G}{\partial x}(0,y) = 0. \]
Therefore, for all $x,y,z,w$ as above, we have
\[ \begin{split}
|G_{z,w} (x,y)| &\leq |G(x,y) - G(z,y)| + |G(x,w) - G(z,w)| \\
&\leq 2 x^{2} \, \underset{\mathbb{R}^{2}}{\sup} \left| \frac{\partial^{2} G}{\partial x^{2}}\right|
\end{split} \]
where the second inequality follows from the fact that $0 \leq z \leq x $ due to our assumptions on $\epsilon$ and $\eta$. But, by Lemma 5.5 in \cite{EladAltman2019}, for all $k \geq 0$, $\frac{\partial^{k}}{\partial x^{k}} \Gamma_{r,s}$ is bounded uniformly by
\[ A_k \det(M)^{-\frac{1+k}{2}}, \] 
where $A_k>0$ depends only on $k$. Therefore, by the Leibniz formula, we deduce that there exists a universal constant $A>0$ such that 
\[\underset{\mathbb{R}^{2}}{\sup} \left| \frac{\partial^{2} G}{\partial x^{2}} \right| \leq A \det(M)^{-2}.\]
Hence, by \eqref{lower_bound_det}, we obtain
\[\underset{\mathbb{R}^{2}}{\sup} \left| \frac{\partial^{2} G}{\partial x^{2}} \right| \leq C_\vartheta \, (t \wedge 1)^{-1}. \]
We thus obtain the bound
\begin{equation}
\label{bound_smallx_bigy}
|G_{z,w} (x,y)| \leq C_{\vartheta} \, x^{2} \, (t \wedge 1)^{-1} 
\end{equation} 
for $x \in [0,1]$. This bound is appropriate in the regime where $y$ is large but $x$ is small. In the same way, we obtain the bound
\begin{equation}
\label{bound_bigx_smally}
|G_{z,w} (x,y)| \leq C_{\vartheta} \, y^{2} \, (t \wedge 1)^{-1}
\end{equation} 
for $y \in [0,1]$, which takes care of the case when $x$ is large but $y$ is small. There remains to obtain a bound for $x$ and $y$ which are both small. To do so, note that
\[ \begin{split}
 |G_{z,w} (x,y)|  &= \left| \int_{w}^{y} \int_{z}^{x} \frac{\partial^{2} G}{\partial x \partial y} (u,v) \d u \d v \right| \\
&\leq x y \, \underset{[0,x] \times [0,y]}{\sup} \left| \frac{\partial^{2} G}{\partial x \partial y} \right|.  
 \end{split}\]
Now, differentiating the expression for $G$ in \eqref{def_g}, we obtain, for all $x,y \geq 0$
\[ \frac{\partial^{2} G}{\partial x \partial y}(x,0) = \frac{\partial^{2} G}{\partial x \partial y}(0,y) = 0.\]
Similarly, we have
\[ \frac{\partial^{3} G}{\partial x^{2} \partial y}(x,0) = \frac{\partial^{3} G}{\partial x \partial y^{2}}(0,y)=0. \]
Therefore, we have, for all $x,y \in [0,1]$
\[ \left| \frac{\partial^{2} G}{\partial x \partial y}(x,y) \right| \leq x y \, \underset{[0,1]^{2}}{\sup} \left| \frac{\partial^{4} G}{\partial x^{2} \partial y^{2}}(x,y) \right|, \]
whence we obtain
\[ \begin{split}
 |G_{z,w} (x,y)| &\leq x^{2} y^{2} \, \underset{[0,1]^{2}}{\sup} \left| \frac{\partial^{4} G}{\partial x^{2} \partial y^{2}} \right| \\
&\leq K \, x^{2} y^{2} \,\det(M)^{-3},     
 \end{split}\]
where the last bound follows from Lemma 5.5 in \cite{EladAltman2019} and the Leibniz formula, with $K>0$ a universal constant. But, interpolating between the lower bounds provided by \eqref{lower_bound_det_trivial} and \eqref{lower_bound_det}, we have, for any $\alpha \in (0,1)$
\[\begin{split}
\det(M)^{-3} &\leq (\vartheta^{2}|r-s|)^{\alpha-1}(c_{\vartheta} \, \sqrt{t \wedge 1})^{-(\alpha+2)} \\
\end{split}\]
Choosing for example $\alpha=1/2$, we obtain 
\[\det(M)^{-3}  \leq C_\vartheta \, |r-s|^{-1/2} \, (t \wedge 1)^{-5/4}.\]
Hence 
\begin{equation}
\label{bound_xsmall_ysmall}
 |G_{z,w} (x,y)| \leq
  C_{\vartheta} \, x^{2} y^{2} \, |r-s|^{-1/2} \, (t \wedge 1)^{-5/4} 
\end{equation}
for all $x,y \in [0,1]$.
%
%
%
%
Finally, we can now bound $I(\epsilon,\eta)$. To do so, we decompose this integral as follows:
\[ I(\epsilon,\eta) = I_{[1,\infty]^{2}} + I_{[\epsilon,1) \times [1,\infty]} + I_{[1,\infty] \times [\epsilon,1)} + I_{[\epsilon,1)^{2}},\]
where, for $A \subset \mathbb{R}_{+}^{2}$, $I_{A}$ denotes the integral
\[  \int_{A} \frac{\d x}{x^{2}}  \frac{\d y}{y^{2}} \, \int_{[0,\eta]^{2}} \d z \d w \, \rho_{\eta}(z) \rho_{\eta}(w) \, G_{z,w} (x,y).  \]
We start by obtaining a bound for $I_{[1,\infty]^{2}}$. By \eqref{bound_bigx_bigy}, and recalling that $\int_0^\eta \rho_{\eta}(x) \d x = 1/2$, we have
\[|I_{[1,\infty]^{2}}| \leq   C_\vartheta \, (t \wedge 1)^{-1/2}. \]  
On the other hand, by \eqref{bound_smallx_bigy}, we have
\[ \begin{split} |I_{[\epsilon,1) \times [1,\infty]}| &\leq C_{\vartheta} \, (t \wedge 1)^{-1} \int_{\epsilon}^{1} \d x \, \int_{1}^{\infty} \frac{\d y}{y^{2}} \\
&\leq C_{\vartheta} \, (t \wedge 1)^{-1}.
\end{split} \] 
Similarly, by \eqref{bound_bigx_smally}, we have
\[ |I_{[1,\infty] \times [\epsilon,1)}| \leq C_{\vartheta} \, (t \wedge 1)^{-1}.\] 
Finally, by \eqref{bound_xsmall_ysmall}, we have
\[ \begin{split} |I_{[\epsilon,1) \times [1,\infty]}| &\leq C_{\vartheta} \, |r-s|^{-1/2} \, (t \wedge 1)^{-5/4} \int_{\epsilon}^{1} \d x \, \int_{\epsilon}^{1} \d y \\
&\leq C_{\vartheta} \, |r-s|^{-1/2} (t \wedge 1)^{-5/4}.
\end{split} \] 
Putting these bounds together finally yields 
\[ |I(\epsilon,\eta)| \leq C_{\vartheta} \, |r-s|^{-1/2} \, (t \wedge 1)^{-5/4} \]  
for all $\eta < \epsilon <1$ and $t >0$. Therefore, recalling \eqref{norm_sgp}, we obtain
\[ \| \mathbf{Q}^2_t G_{\epsilon, \eta} \|^2_{L^2(\mu_2)} \leq C_\vartheta \, (t \wedge 1)^{-5/4} \, \|h \|_{\infty}^2 \int_{[0,1]^2} |r-s|^{-1/2} \d r \, \d s, \]
where the last integral is finite. Hence
\[ \| \mathbf{Q}^2_t G_{\epsilon, \eta} \|_{L^2(\mu_2)} \leq C(\vartheta,h) \, (t \wedge 1)^{-5/8}, \]
where $C(\vartheta,h)$ is a constant which depends only on $\vartheta$ and $h$. Therefore
\[ \| V_{\epsilon, \eta} \|_{L^2(\mu_2)} \leq C(\vartheta,h) \, \int_0^\infty e^{-t} \, (t \wedge 1)^{-5/8} \d t < \infty.\]
Thus, $(V_{\epsilon, \eta})_{0< \eta < \epsilon < 1}$ is bounded in $L^2(\mu_2)$. Reasoning similarly to bound $V_{\epsilon, \eta}-V_\epsilon$ and $V_\epsilon - V$, we deduce by dominated convergence that 
\[V_{\epsilon, \eta} \underset{\eta \to 0}{\longrightarrow} V_{\epsilon} \]
and
\[V_{\epsilon} \underset{\epsilon \to 0}{\longrightarrow} V \] 
in $L^2(\mu_2)$. 
There remains to show that these convergences hold in $D(\Lambda^2)$. To do so, for all $\eta < \epsilon <1$, we bound $\|\overline{\nabla} V_{\epsilon,\eta}\|_{L^2}$, where $\| \cdot \|_{L^2}$ stands for the norm in $L^2(\mathbb{H}_2,\mu_2; \mathbb{H}_2)$. We have
\[\begin{split}
&\|\overline{\nabla} Q^2_t G_{\epsilon,\eta}\|^2_{L^2} = \int_{[0,1]^2} \d r \, \d s \, h_r \, h_s \int_{\mathbb{R}^2\times \mathbb{R}^2} \d a \, \d b \, f_{\epsilon,\eta}(\|a\|) \, f_{\epsilon,\eta}(\|b\|) \\
& \left(\frac{\partial^2 \Gamma_{r,s}}{\partial x \, \partial y}(a_1,b_1) \, \Gamma_{r,s}(a_2,b_2) + \Gamma_{r,s}(a_1,b_1) \, \frac{\partial^2 \Gamma_{r,s}}{\partial x \, \partial y}(a_2,b_2)\right).
\end{split}\]
Therefore it suffices to bound, for all $(r,s) \in [\vartheta,1-\vartheta]^{2}$, the integral
\[ I^{(1)}(\epsilon,\eta) :=  \int_{\mathbb{R}_+ \times \mathbb{R}_+} f_{\epsilon, \eta}(x) f_{\epsilon, \eta}(y) \, x y \, G^{(1)}(x,y) \d x \d y,\]
where 
\[\begin{split}
G^{(1)}(x,y) :=&  \int_{[0,2 \pi]^{2}} \Big( \frac{\partial^2 \Gamma_{r,s}}{\partial x \, \partial y} (x \cos(\theta), y \cos(\varphi)) \, \Gamma_{r,s}(x \sin(\theta), y \sin(\varphi)) \\
&+ \Gamma_{r,s}(x \cos(\theta), y \cos(\varphi)) \, \frac{\partial^2 \Gamma_{r,s}}{\partial x \, \partial y} (x \sin(\theta), y \sin(\varphi)) \Big) \d \theta \d \varphi. 
\end{split}\]
Reasoning as above, we can rewrite $I^{(1)}(\epsilon,\eta)$ as 
\[\begin{split}
I(\epsilon,\eta) =  (16)^2 \int_{[\epsilon, \infty) \times [\epsilon, \infty)} \frac{\d x}{x^{2}}  \frac{\d y}{y^{2}} \, \int_{[0,\eta]^{2}} \d z \d w \, \rho_{\eta}(z) \, \rho_{\eta}(w) \, G^{(1)}_{z,w} (x,y), 
\end{split}\]
where
\[G^{(1)}_{z,w} (x,y) := G^{(1)}(x,y) - G^{(1)}(x,w) - G^{(1)}(z,y) + G^{(1)}(z,w). \]
By the triangular inequality we have, for all $x,y,z,w \in \mathbb{R}$
\[|G^{(1)}_{z,w} (x,y)| \leq 4 \, \underset{\mathbb{R}_+^2}{\sup} |G^{(1)}|. \]
Now, the supremum above is bounded by
\[ 2 \, \underset{\mathbb{R}^2}{\sup} \left| \frac{\partial^2}{\partial x \, \partial y} \Gamma_{r,s} \right| \,  \underset{\mathbb{R}^2}{\sup} \left| \Gamma_{r,s} \right|.\]
By Lemma 5.5 in \cite{EladAltman2019}, this in turn is bounded by
\[K \det(M)^{-2} \]
for some universal constant $K>0$. In virtue of \eqref{lower_bound_det}, we thus deduce that
\[|G^{(1)}_{z,w} (x,y)| \leq C_\vartheta \, (t\wedge1)^{-1}. \]
On the other hand, noting that, for all $y \in \mathbb{R}$
\[\frac{\partial G^(1)}{\partial x}(0,y) = 0 \]
we have, for all $x \in [0,1]$ and $y,z,w  \in \mathbb{R}_+$, the bound
\[ |G^{(1)}_{z,w} (x,y)| \leq 2 x^2 \, \underset{\mathbb{R}_+^2}{\sup} \left| \frac{\partial^2 G^(1)}{\partial x^2} \right|. \]
But, by Lemma 5.5 in \cite{EladAltman2019}, the Leibniz formula,  and the bound \eqref{lower_bound_det}, we have
\[\underset{\mathbb{R}_+^2}{\sup} \left| \frac{\partial^2 G^(1)}{\partial x^2} \right| \leq C_\vartheta \, (t \wedge 1)^{-3/2}. \]  
Hence, when $x \in [0,1]$, we have the bound
\[|G^{(1)}_{z,w} (x,y) | \leq C_\vartheta \, (t\wedge1)^{-3/2} \, x^2.\] 
Similarly, when $y \in [0,1]$, we have
\[|G^{(1)}_{z,w} (x,y) | \leq C_\vartheta \, (t\wedge1)^{-3/2} \, y^2.\]
Finally, when $x,y \in [0,1]^2$, one has
\[ \begin{split}
 |G^{(1)}_{z,w} (x,y)| &\leq x^{2} y^{2} \, \underset{\mathbb{R}_+^2}{\sup} \left| \frac{\partial^{4} G^{(1)}}{\partial x^{2} \partial y^{2}} \right| \\ 
&\leq K \, x^{2} y^{2} \, \det(M)^{-4},  
\end{split}\]
where the last bound follows from Lemma 5.5 in \cite{EladAltman2019} and the Leibniz formula, with $K>0$ a universal constant. 
Now, by interpolation of \eqref{lower_bound_det_trivial} and \eqref{lower_bound_det}, we have
\[  \det(M)^{-4} \leq (\vartheta^{2}|r-s|)^{-1/2}(c_{\vartheta} \sqrt{t \wedge 1})^{-7/2}. \]
Therefore, for all $x,y \in [0,1]$, we have
\[ |G^{(1)}_{z,w} (x,y)| \leq C_\vartheta \, (t \wedge 1)^{-7/4} \, |r-s|^{-1/2} \, x^2 y^2. \]
We now put together all these estimates by writing
\[ I^{(1)}(\epsilon,\eta) = I^{(1)}_{[1,\infty]^{2}} + I^{(1)}_{[\epsilon,1) \times [1,\infty]} + I^{(1)}_{[1,\infty] \times [\epsilon,1)} + I^{(1)}_{[\epsilon,1)^{2}},\]
where, for $A \subset \mathbb{R}_{+}^{2}$, $I^{(1)}_{A}$ denotes the integral
\[  \int_{A} \frac{\d x}{x^{2}}  \frac{\d y}{y^{2}} \, \int_{[0,\eta]^{2}} \d z \d w \, \rho_{\eta}(z) \rho_{\eta}(w) \, G^{(1)}_{z,w} (x,y).  \]
The previous estimates yield
\[I^{(1)}_{[1,\infty]^{2}} \leq C_\vartheta \, (t\wedge1)^{-1}\]
as well as
\[I^{(1)}_{[\epsilon,1) \times [1,\infty]} \leq  C_\vartheta \, (t\wedge1)^{-3/2}, \quad I^{(1)}_{[1,\infty] \times [\epsilon,1)} \leq C_\vartheta \, (t\wedge1)^{-3/2},\]
and
\[ I^{(1)}_{[\epsilon,1)^{2}} \leq C_\vartheta \, (t \wedge 1)^{-7/4} \, |r-s|^{-1/2}. \]
Therefore, we obtain 
\[ I^{(1)}(\epsilon,\eta) \leq C_\vartheta \, (t \wedge 1)^{-7/4} \, |r-s|^{-1/2}\]
and, since $ \int_{[0,1]^2} |r-s|^{-1/2} \d r \, \d s < \infty$, we deduce finally that 
\[  \|\overline{\nabla} Q^2_t G_{\epsilon,\eta}\|^2_{L^2} \leq C(\vartheta,h) \, (t \wedge 1)^{-7/4}, \]
where $C(\vartheta,h) \in (0,\infty)$ depends only on $\vartheta$ and $h$. Therefore, 
\[\| \overline{\nabla} V_{\epsilon, \eta} \|_{L^2} \leq \sqrt{C(\vartheta,h)} \, \int_0^\infty e^{-t} (t \wedge 1)^{-7/8} \d t < \infty. \]
Hence, $(\overline{\nabla} V_{\epsilon, \eta})_{\epsilon > \eta> 0}$ is uniformly bounded in $L^2(\mathbb{H}_2,\mu_2; \mathbb{H}_2)$. Similar bounds on $\nabla V_{\epsilon,\eta} - \nabla V_{\epsilon} $ and on $\nabla V_{\epsilon} - \nabla V$ yield, by dominated convergence, 
\[\overline{\nabla} V_{\epsilon, \eta} \underset{\eta \to 0}{\longrightarrow} \nabla V_{\epsilon} \]
and
\[\overline{\nabla} V_{\epsilon} \underset{\epsilon \to 0}{\longrightarrow} \nabla V \] 
in $L^2(\mathbb{H}_2,\mu_2; \mathbb{H}_2)$. The proposition is proved.
\end{proof} 

\bibliographystyle{abbrvnat}
\bibliography{Bessel_SPDEs_with_general_boundary_conditions}

\end{document}